\documentclass[12pt]{article}
\usepackage{amsfonts}
\usepackage{mathrsfs,dsfont,bbm}

\usepackage[latin1]{inputenc}
\usepackage{amsmath,amssymb}
\usepackage{latexsym}
\usepackage{epstopdf}
\usepackage{geometry}
\usepackage[active]{srcltx}
\usepackage[
bookmarks=true,         
bookmarksnumbered=true, 
colorlinks=true, pdfstartview=FitV, linkcolor=blue, citecolor=blue,
urlcolor=blue]{hyperref}

\newtheorem{theorem}{Theorem}[section]

\newtheorem{lemma}[theorem]{Lemma}
\newtheorem{proposition}[theorem]{Proposition}
\newtheorem{definition}[theorem]{Definition}

\newtheorem{hypothesis}[theorem]{Hypothesis}

\newtheorem{example}[theorem]{Example}
\numberwithin{equation}{section}
\newenvironment{proof}[1][Proof]{\noindent\textit{#1.} }{\hfill \rule{0.5em}{0.5em}}

\begin{document}
	
	\title{Large deviation principle for multi-scale distribution dependent stochastic differential equations driven by fractional Brownian motions }
	\date{\today}
	
	\author {Guangjun Shen, Huan Zhou and Jiang-Lun Wu\thanks{Corresponding author. This research is supported by the
			National Natural Science Foundation of China (12071003)}}

\maketitle

\begin{abstract}
		In this paper, we are concerned with multi-scale distribution dependent stochastic differential equations driven by fractional Brownian motion (with Hurst 
	index $H>\frac12$) and standard Brownian motion, simultaneously. Our aim is to establish a large deviation principle for the multi-scale distribution dependent stochastic differential equations. This is done via the weak convergence approach and our proof is based heavily on the fractional calculus. 
	
\end{abstract}

\section{Introduction}\label{sec1}
As is well known, stochastic differential equations (SDEs) play a significant role in modelling evolutions of dynamical systems when taking into account uncertainty features
in diverse fields ranging from biology, chemistry, and physics, as well as ecology, economics and finance, and so on (see, for example,
Sobczyk \cite{Kazimier} and the references therein).
Generally, nonlinear Fokker-Planck equations can be characterised by distribution dependent stochastic differential equations (DDSDEs), which are also named as McKean-Vlasov SDEs or mean field SDEs. DDSDEs could be used to describe stochastic systems whose evolution is influenced by both the microcosmic location and the macrocosmic distribution of particles, i.e., the coefficients of equations depend not only on the solution itself but also on its time marginal law. As such, there have been many applications of DDSDEs in numerous fields such as statistical physics, mean-field games, mathematical finance and biology (see, for example, Buckdahn et al. \cite{Buckdahn}, Bossy and Talay \cite{Bossy}, Carmona and Delarue \cite{Carmona}, and the references therein). Recently, there has been an increasing interest in studying existence and uniqueness for solutions of DDSDEs. Wang \cite{Wang} established strong well-posedness of DDSDEs with one-sided Lipschitz continuous drifts and Lipschitz-continuous dispersion coefficients. Under integrability conditions on distribution dependent coefficients, Huang and Wang \cite{Huang1} obtained the existence and uniqueness for DDSDEs with non-degenerate noise.  Mehri and Stannat \cite{Mehri} proposed a Lyapunov-type approach to the problem of existence and uniqueness of general DDSDEs. Many interesting studies of DDSDEs have been developed further in Bao et al. \cite{Bao}, 
Ren and Wang \cite{Ren}, R\"{o}ckner and Zhang \cite{Rockner},  Mishura and Veretennikov \cite{Mishura}, Chaudru and Raynal \cite{Chaudru}, Hammersley et al. \cite{Hammersley}, just mention a few. Although there exist many investigations in the literature devoted to studying DDSDEs driven by Brownian motion, L\'evy processes, as we know, there is few consideration   for DDSDEs driven   by fractional Brownian motion (fBm) which is neither a Markov process nor a semimartingale.  Fan et al. \cite{Fan} considered the following DDSDEs driven by fBm with Hurst parameter $H>\frac12$
\begin{equation*}\label{sec1-eq1.1}
	dX_{t}=b(t,X_{t},\mathscr{L}_{X_{t}})dt+\sigma(t,\mathscr{L}_{X_{t}})dB^{H}_t
\end{equation*}
by showing the well-posedness and derived a Bismut type formula for the Lions derivative using Malliavin calculus. Galeati et al. \cite{Galeati} studied DDSDEs with irregular, possibly distributional drift, driven by additive fBm of Hurst parameter $H\in (0,1)$ and established strong well-posedness under a variety of assumptions on the drifts. Shen, Xiang and Wu \cite{Shen1} studied averaging principle of DDSDEs driven simultaneously by fBm with Hurst index $H>\frac12$ and standard Brownian motion under certain averaging conditions.
Buckdahn et al. \cite{Buckdahn} considered mean-field SDEs driven by fBm and related stochastic control problems. Bauer and Meyer-Brandis \cite{Bauer} established existence and uniqueness results of solutions to McKean-Vlasov equations driven by cylindrical fBm in an infinite-dimensional Hilbert space setting with irregular drift.

In this paper, we will study the following system of multi-scale DDSDEs with small fractional noises
\begin{equation}\label{sec1-eq1.2}
	\left\{
	\begin{aligned}
		&dX_{t}^{\delta,\epsilon}=b(t,X_{t}^{\delta,\epsilon},\mathscr{L}_{X_{t}^{\delta,\epsilon}},Y_{t}^{\delta,\epsilon})dt+\delta^{H}\sigma(t,\mathscr{L}_{X_{t}^{\delta,\epsilon}})dB_{t}^{H},\\
		&dY_{t}^{\delta,\epsilon}=\frac{1}{\epsilon}f(t,X_{t}^{\delta,\epsilon},\mathscr{L}_{X_{t}^{\delta,\epsilon}},Y_{t}^{\delta,\epsilon})dt+\frac{1}{\sqrt{\epsilon}}g(t,X_{t}^{\delta,\epsilon},\mathscr{L}_{X_{t}^{\delta,\epsilon}},Y_{t}^{\delta,\epsilon})dW_{t},\\
		&X_{0}^{\delta,\epsilon}=x\in \mathbb{R}^{n},\qquad Y_{0}^{\delta,\epsilon}=y\in \mathbb{R}^{m},
	\end{aligned}
	\right.
\end{equation}
where $\mathscr{L}_{X_{t}^{\delta,\epsilon}}$ denotes the law of $X_{t}^{\delta,\epsilon}$, $\delta$ and $\epsilon$ are the scale parameters satisfying certain assumptions specified in the sequel (see Theorem \ref{sec3-th3.5} below), $\epsilon$ is a small positive parameter describing the
ratio of the time scale between the slow component $X_{t}^{\delta,\epsilon}$ and fast component $Y_{t}^{\delta,\epsilon}$, $\{B_{t}^{H}\}_{t\geq0}$ and $\{W_{t}\}_{t\geq0}$ are mutually independent $n$-dimensional fBm with Hurst parameter $H\in(1/2,1)$ and $m$-dimensional standard Brownian motions on a given complete filtered probability space $(\Omega, \mathscr{F},\{\mathscr{F}_{t}\}_{t\geq0},\mathbb{P})$, respectively. More precisely, we take $\Omega$ to be the Banach space $C_{0}([0,T];\mathbb{R}^{n})$ of continuous functions vanishing at 0 equipped with the supremum norm, $\mathscr{F}$ is the Borel $\sigma$-algebra,  $\{\mathscr{F}_{t}\}_{t\geq0}$ denotes the filtration generated by $B^{H}$ and $W$,  and $\mathbb{P}$ is the unique canonical probability measure on $\Omega$.

Due to the different time scales and the cross interactions between the fast and slow
components, it is very difficult to analyse such kind of stochastic system directly.
More recently, under ${\bf(H_1)}$ and ${\bf(H_2)}$ (see Sec \ref{sec2} below), Shen, Yin and Wu \cite{Shen2}  showed   the slow component $X_{t}^\epsilon$ (taking $\delta=1$ in Eq. \eqref{sec1-eq1.2}) strongly converges to the solution $\bar{X}$  of the associated averaged equations
\begin{equation}\label{sec1-eq1.12}
	\left\{
	\begin{aligned}
		&d\bar{X}_{t}=\bar{b}(t,\bar{X}_{t},\mathscr{L}_{\bar{X}_{t}})dt
		+\sigma(t,\mathscr{L}_{\bar{X}_{t}})dB_{t}^{H},\\
		&\bar{X}_{0}=x,
	\end{aligned}
	\right.
\end{equation}
where
\begin{equation}\label{sec1-eq1.31}
	\bar{b}(t,x,\mu)=\int_{\mathbb{R}^{m}}b(t,x,\mu,z)\nu^{t,x,\mu}(dz),
\end{equation}
and $\nu^{t,x,\mu}$ is the unique invariant measure for the transition semigroup of the solution of the following frozen equations

\begin{equation*}\label{sec1-eq1.3}
	\left\{
	\begin{aligned}
		&dY_{s}=f(t,x,\mu,Y_{s})ds+g(t,x,\mu,Y_{s})d\tilde{W}_{s},\\
		&Y_{0}=y,
	\end{aligned}
	\right.	
\end{equation*}
where $\tilde{W}_{t}$ is a $m$	dimensional Brownian motion on another given complete probability space $(\tilde{\Omega},\tilde{\mathscr{F}},\{\tilde{\mathscr{F}_{t}}\}_{t\geq0},\mathbb{\tilde{P}})$ and $\{\tilde{\mathscr{F}_{t}}\}_{t\geq0}$ is the natural filtration generated by $\tilde{W}_{t}$.

However, the averaged process $\bar{X}$ is valid only in the limiting sense, and it is clear that
the slow process $X_{t}^\epsilon$  will experience fluctuations around its averaged process $\bar{X}$ for
small $\epsilon$. In order to
capture the fluctuations, it is important to study the asymptotic behavior of the deviation
between $X_{t}^\epsilon$  and $\bar{X}$. Large  deviation principles (LDP) are to calculate the probability of a rare event, which investigate the asymptotic property of remote tails of a family of a probability distribution. In the case of stochastic processes, the idea lies in identifying a deterministic path around which the diffusion is concentrated with high probability  which
leads to a interpretation of the stochastic motion as a small perturbation of this deterministic
path. A powerful   approach for studying large deviation problems is the well-known weak convergence method (see, for example,  Budhiraja and  Dupuis \cite{Budhiraja1}, Matoussi, Sabbagh and   Zhang \cite{Matoussi}, Ren and Zhang \cite{Ren}), this approach has been widely applied in various
stochastic dynamical systems driven by Brownian motion, L\'{e}vy process or fBm.
Dupuis and Spiliopoulos \cite{Dupuis} studied the LDP for locally periodic SDEs with small noise and fast ascillating coefficients.
Bezemek and Spitiopoulos \cite{Bezemek} made use of weak convergence methods providing a convenient representation for the large deviations rate function.
Liu et al. \cite{Liu} considered McKean-Vlasov SDEs driven by L\'{e}vy noise and applied the weak convergence method to establish large and moderate deviation principles.
Budhiraja and Song \cite{Budhiraja} studied small noise large deviations asymptotics for SDEs with a multiplicative noise given as a fBm. Fan et al. \cite{Fan} studied small-time asymptotic behaviors for a class of DDSDEs driven by fBm with
Hurst parameter $H\in(1/2, 1)$ and magnitude $\delta^{H}$ and established the LDP for this type equations.
Many interesting studies of  LDP have been developed  further in
Brzeniak, Goldys and  Jegaraj \cite{Brzeniak},
Budhiraja et al. \cite{BudhirajaDupuis}, Dong et al. \cite{Dong}, Hong et al. \cite{HLL},
Matoussi et al. \cite{Matoussi},
Suo and Yuan \cite{Suo},
Wang et al. \cite{WangZhai} 
and the references therein.
In the distribution independent case, there have been many fundamental studies addressing
the LDP for two-time-scale stochastic systems driven by Brownian motion, jump process and fBm.
Hong et al. \cite{HongLiLiu} considered Freidlin-Wentzell type LDP for multi-scale locally monotone stochastic partial differential equations.   Kumar  and  Popovic \cite{Kumar} studied LDP for multi-scale jump-diffusion processes. Sun et al. \cite{Sun} obtained LDP for two-time-scale stochastic Burgers equations. Bourguin et al. \cite{Bourguin2} studied typical dynamics and fluctuations for a slow-fast dynamics system perturbed by a small fractional Brownian noise.
Gailus and Gasteratos \cite{Gailus} considered the LDP for a multiscale system of stochastic differential equations in which the slow component is perturbed by
a small fBm in the homogenized limit.
Inahama et al. \cite{Inahama} established the LDP for slow-fast system with mixed fBm.

It is worth noting that there is few LDP result for multi-scale distribution dependent stochastic system
so far. To the best of our knowledge, Hong et al.  \cite{Hong} is the first result concerning the LDP for multi-scale   DDSDEs based on the techniques of weak convergence approach.
Although there exist some investigations in the literature devoted to studying LDP for multi-scales SDEs or DDSDEs driven by Brownian motion, L\'{e}vy processes and fBm. However, there is not any consideration of LDP for multi-scales DDSDEs driven by fBm. It is interesting to find how the multi-scales influence Eq. (\ref{sec1-eq1.2}) when $\delta$ and $\epsilon$ converge to zero simultaneously. The main purpose of this paper is to consider the small noise asymptotic behavior and establish LDP for Eq. (\ref{sec1-eq1.2}). More precisely,  we will prove that $\{X^{\delta,\epsilon}\}_{\delta>0}$ in Eq. (\ref{sec1-eq1.2}) satisfies the LDP in $C([0,T];\mathbb{R}^{n})$ as $\delta\rightarrow0$. It is worth stressing that compared with the works
in the Brownian motion and L\'{e}vy processes cases, there are substantial new difficulties presented by our setting since fBm  is neither a Markov process nor a semimartingale, so some techniques based on  It\^{o} calculus are not applicable. Our strategy in this paper is based on fractional calculus and weak convergence approach. The time
discretization technique will also be employed frequently to obtain some crucial estimates.

		The rest of this paper is organized as follows. Section \ref{sec2} contains some necessary preliminaries of fBm and some basic properties. In Section \ref{sec3}, we present the main result concerning the LDP. In Section \ref{sec4}, we devote to proving the main result.
		Throughout this paper, the letter $C$ will denote a positive constant, with or without
		subscript, its value may change in different occasions. We will write the dependence of the
		constant on parameters explicitly if it is essential. Let $|\cdot|$ and $\left\langle \cdot, \cdot\right\rangle $ be the Euclidean norm and inner product, respectively, and for a matrix,  $\|\cdot\|$ denotes the operator norm.


		\section{ Preliminaries}\label{sec2}
		
		As an extension of Brownian motion, the fBm  exhibits long-range dependence and self-similarity, having stationary increments. It is the usual candidate to model
		phenomena in which the self-similarity property can be observed from
		the empirical data.
		Recall that the fBm $B^{H}=(B^{H,1},\cdots,B^{H,n})$  with Hurst index $H\in
		(\frac12,1)$  is a centered Gaussian process, whose covariance structure is defined by
		$$\mathbb{E}(B_{t}^{H,i}B_{s}^{H,j})=R_{H}(t,s)\delta_{i,j},\quad s,t\in[0,T],\quad i,j=1,\cdots,n$$
		with $R_{H}(t,s)=\frac{1}{2}(t^{2H}+s^{2H}-|t-s|^{2H})$.
		Thus, Kolmogorov's continuity criterion implies that fBm is H\"{o}lder continuous of order $\delta$ for any $\delta<H.$
		Besides, $R_{H} (t, s)$ has the following integral representation
		$$R_{H}(t,s)=\int_{0}^{t\wedge s}K_{H}(t,r)K_{H}(s,r)dr,$$
		where the deterministic kernel $K_{H}(t,s)$ is   given by
		$$K_{H}(t,s)=C_{H}s^{\frac{1}{2}-H}\int_{s}^{t}(u-s)^{H-\frac{3}{2}}u^{H-\frac{1}{2}}du,\quad t>s,$$
		where $C_{H}=\sqrt{\frac{H(2H-1)}{\mathcal{B}(2-2H,H-1/2)}}$ and $\mathcal{B}$ standing for the Beta function. If $t\leq s$, we set $K_{H}(t,s)=0$.
		FBm was first
		introduced by Kolmogorov and studied by Mandelbrot and Van Ness
		\cite{man}, where a stochastic integral representation in terms of a
		standard Brownian motion was established.
		For
		$H=\frac1 2$, $B^H$ coincides with the standard Brownian motion $B$, but
		$B^H$ is neither a semimartingale nor a Markov process unless
		$H=\frac 12$. As a consequence, some classical techniques of stochastic
		analysis are not applicable. Interesting surveys of fBm and related stochastic calculus
		could be found in  Biagini et al. \cite{ Biagini}
		and references therein.

		In the follows, we recall the basic definitions and properties of the fractional calculus. For a detailed presentation of these notions we refer Samko et al. \cite{{Samko}}. Let $a,b\in\mathbb{R}$, $a<b$.  Let $f\in L^{1}(a,b)$ and $\alpha>0$. The left and right-sided fractional integrals of $f$ of order $\alpha$ are defined for almost all $x\in(a,b)$ by
		$$I_{a+}^{\alpha}f(x)=\frac{1}{\Gamma(\alpha)}\int_{a}^{x}(x-y)^{\alpha-1}f(y)dy,$$
		and
		$$I_{b-}^{\alpha}f(x)=\frac{1}{\Gamma(\alpha)}\int_{x}^{b}(y-x)^{\alpha-1}f(y)dy,$$
		respectively. Let $f\in I_{a+}^{\alpha}(L^{p})$ (resp. $f\in I_{b-}^{\alpha}(L^{p})$) and $0<\alpha<1$, then the left and right-sided fractional derivatives are defined by
		$$D_{a+}^{\alpha}f(x)=\frac{1}{\Gamma(1-\alpha)}\Big(\frac{f(x)}{(x-a)^{\alpha}}+\alpha\int_{a}^{x}\frac{f(x)-f(y)}{(x-y)^{\alpha+1}}dy\Big),$$
		and
		$$D_{b-}^{\alpha}f(x)=\frac{1}{\Gamma(1-\alpha)}\Big(\frac{f(x)}{(b-x)^{\alpha}}+\alpha\int_{x}^{b}\frac{f(x)-f(y)}{(y-x)^{\alpha+1}}dy\Big),$$
		for almost all $x\in(a,b)$ (the convergence of the integrals at the singularity $y=x$ holds pointwisely for almost all $x\in(a,b)$ if $p=1$ and moreover in $L^{P}$ sense if $1<p<\infty$).
		
		Recall the following properties of these operators:
		
		$\bullet$ If $\alpha<\frac{1}{p}$ and $q=\frac{p}{1-\alpha p}$, then
		$$
		I_{a+}^\alpha(L^p)=I_{b-}^\alpha(L^p)\subset L^q(a,b).
		$$
		
		$\bullet$ If $\alpha>\frac{1}{p}$, then
		$$
		I_{a+}^\alpha(L^p)\cup I_{b-}^\alpha(L^p)\subset C^{\alpha-\frac{1}{p}}(a,b),
		$$
		where $C^{\alpha-\frac{1}{p}}(a,b)$ denotes the space of $(\alpha-\frac{1}{p})$-H\"{o}lder continuous functions of order $\alpha-\frac{1}{p}$ in the interval $[a,b]$.
		
		The following inversion formulas hold:
		$$
		I_{a+}^\alpha(D^\alpha_{a+}f)=f
		$$
		for all $f\in I_{a+}^\alpha(L^p)$, and
		$$
		D_{a+}^\alpha (I_{a+}^\alpha f)=f
		$$
		for all $f\in L^1(a,b)$. Similar inversion formulas hold for the operators $I_{b-}^\alpha$ and $D_{b-}^\alpha$.
		
		The following integration by parts formula holds:
		$$
		\int_a^b(D^\alpha_{a+}f)(s)g(s)ds=\int_a^bf(s)(D^\alpha_{b-}g)(s)ds,
		$$
		for any $f\in I_{a+}^\alpha(L^p)$, $g\in I_{b-}^\alpha(L^q)$, $\frac{1}{p}+\frac{1}{q}=1$.
		To prove our main results, we also present the following Hardy-Littlewood inequality.
		\begin{lemma}\label{sec3-lem3.1}(\cite{Stein})
			Let $1<\tilde{p}<\tilde{q}<\infty$ and $\frac{1}{\tilde{q}}=\frac{1}{\tilde{p}}-\alpha$. If $f:\mathbb{R}^+\rightarrow\mathbb{R}$ belongs to $L^{\tilde{p}}(0,\infty)$, then $I_{0+}^\alpha f(x)$ converges absolutely for almost every x, and moreover
			$$
			\|I_{0+}^\alpha f(x)\|_{L^{\tilde{q}}(0,\infty)}\leq C_{\tilde{p},\tilde{q}}\|f\|_{L^{\tilde{p}}(0,\infty)}
			$$
			holds for some positive constant $C_{\tilde{p},\tilde{q}}$.
		\end{lemma}
		
		Consider the operator $K_{H}$ induced by the kernel $K_H(t,s)$ as follows:
		$$K_{H}:L^{2}([0,T];\mathbb{R}^{n})\rightarrow I_{0+}^{H+1/2}(L^{2}([0,T];\mathbb{R}^{n}))$$
		by
		$$(K_{H}f)(t):=\int_{0}^{t}K_{H}(t,s)f(s)ds.$$
		On the other hand, for $H> 1/2$, the operator $K_{H}$ can be represented as
		$$K_{H}f:=C_{H} \Gamma(H-1/2)I_{0+}^{1}t^{H-1/2}I_{0+}^{H-1/2}t^{1/2-H}f.$$
		Besides, we denote by $\dot{K}_{H}$ the ``derivative" of the operator $K_{H}$, i.e.,
		$$\dot{K}_{H}f:=C_{H}\Gamma(H-1/2)t^{H-1/2}I_{0+}^{H-1/2}t^{1/2-H}f.$$
		Then Cameron-Martin space $\mathcal{H}_{H}$ associated with the process $B^{H}_{\cdot}$ is defined by
		$$\mathcal{H}_{H}=\Big\{K_{H}\hat{f}:\hat{f}\in L^{2}([0,T];\mathbb{R}^{n})\Big\}$$
		equipped with the inner product $\left\langle f,g \right\rangle_{\mathcal{H}_{H}}=\left\langle \hat{f},\hat{g} \right\rangle _{L^{2}([0,T];\mathbb{R}^{n})}.$
		It is well known that
		$$\left\langle f,g \right\rangle _{\mathcal{H}_{H}}=H(2H-1)\int_{0}^{T}\int_{0}^{T}|t-s|^{2H-2}\left\langle f(s), g(s) \right\rangle_{\mathbb{R}^{n}}dsdt,$$
		this yields for any $f\in L^2([0,T];\mathbb{R}^{n})$,
		$$\|f\|^{2}_{\mathcal{H}_{H}}\leq 2HT^{2H-1}\|f\|^{2}_{ L^2([0,T];\mathbb{R}^{n})}.$$
		Note that in this paper the noise process contains $B^{H}_{\cdot}$ and $W_{\cdot}$  of the form
		$\Big\{(B^{H}_{t},W_{t}): t\in[0,T]\Big\}.$
		Thus we need to define the Cameron-Martin space (see Bourguin et al. \cite{Bourguin}) associated with $(B^{H}_{\cdot},W_{\cdot})$ given by
		$$\mathcal{H}=\Big\{ (K_{H}\hat{f}_{1},K_{1/2}\hat{f}_{2}): (\hat{f}_{1},\hat{f}_{2})\in L^{2}([0,T];\mathbb{R}^{n+m}) \Big\}.$$
		
		As a Cameron-Martin space, $\mathcal{H}$ is a Hilbert space equipped with the inner product given by
		$$\langle(f_{1},f_{2}),(g_{1},g_{2})\rangle_{\mathcal{H}}=\langle f_{1},g_{1}\rangle_{\mathcal{H}_{H}}+\langle f_{2},g_{2}\rangle_{\mathcal{H}_{1/2}}.$$
		The following result consider the differentiability of elements in $\mathcal{H}_{H}$
		which  will be used throughout the paper.
		\begin{lemma}\label{sec2-lemma2.0}(\cite{Bourguin})
			If $H>1/2$ and $u\in\mathcal{H}_{H}$ such that $u=K_{H}\hat{u}$, $\hat{u}\in L^{2}([0,T];\mathbb{R}^{n})$, then we have
			\begin{equation*}
				\begin{aligned}
					\dot{u}(t)=\dot{K}_{H}\hat{u}(t)&=C_{H}\Gamma(H-1/2)t^{H-1/2}I_{0+}^{H-1/2}t^{1/2-H}\hat{u}(t)\\
					&=C_{H}t^{H-1/2}\int_{0}^{t}(t-s)^{H-3/2}s^{1/2-H}\hat{u}_{s}ds.
				\end{aligned}
			\end{equation*}
		\end{lemma}
		
		$\bullet$ It is obvious that if $H=1/2$ and $v\in\mathcal{H}_{1/2}$, then $\dot{v}_{t}=\hat{v}_{t}$.
		
		$\bullet$ The map $\dot{K}_{H}$  is a bound operator in $L^{2}([0,T];\mathbb{R}^{n})$, which implies that
		$$\|\dot{K}_{H}f\|_{L^{2}([0,T];\mathbb{R}^{n})}\leq C_{H}\|f\|_{L^{2}([0,T];\mathbb{R}^{n})}.$$

		Furthermore, denote $C_{b}(\mathscr{E})$ by the set of all bounded continuous functions $f:\mathscr{E}\rightarrow\mathbb{R}$ with the norm $\|f\|_{\infty}:=\sup_{x\in\mathscr{E}}|f(x)|$, where $\mathscr{E}$ is a Polish space with the Borel $\sigma$-field $\mathcal{B}(\mathscr{E})$.
		Let
		$$\mathcal{A}:=\Big\{\phi:\phi \textit{ is } \mathbb{R}^{n+m} \textit{-valued } \mathscr{F}_{t} \textit{-predictable process and } \|\phi\|^{2}_{\mathcal{H}}<\infty,  \mathbb{P} \textit{-a.s.}\Big\},$$
		and for each $M>0$, let
		$$S_{M}:=\Big\{h\in \mathcal{H}:\frac{1}{2}\|h\|^{2}_{\mathcal{H}}\leq M\Big\},$$
		where $h=(K_{H}\hat{u},K_{1/2}\hat{v})=(u,v)\in\mathcal{H}.$
		That is to say,
		$$\|h\|_{\mathcal{H}}^{2}=\int_{0}^{t}|\hat{u}_{s}|^{2}+|\hat{v}_{s}|^{2}ds=\|u_{\cdot}\|^{2}_{\mathcal{H}_{H}}+\|v_{\cdot}\|^{2}_{\mathcal{H}_{1/2}}<\infty.$$
		It is obvious that $S_{M}$ endowed with the weak topology is a Polish space. Besides, define
		$$\mathcal{A}_{M}:=\Big\{\phi\in\mathcal{A}:\phi(\cdot)\in S_{M}, \mathbb{P}\textit{ -a.s.}\Big\}.$$
		
		We assume that  $b,\sigma,f,g$
		\begin{align*}
			&b:[0,T]\times\mathbb{R}^{n}\times\mathscr{P}_{\theta}(\mathbb{R}^{n})\times\mathbb{R}^{m}\rightarrow\mathbb{R}^{n},~
			~~\sigma:[0,T]\times\mathscr{P}_{\theta}(\mathbb{R}^{n})\rightarrow\mathbb{R}^{n}\otimes\mathbb{R}^{n},\\
			&f:[0,T]\times\mathbb{R}^{n}\times\mathscr{P}_{\theta}(\mathbb{R}^{n})\times\mathbb{R}^{m}\rightarrow\mathbb{R}^{m},~
			~g:[0,T]\times\mathbb{R}^{n}\times\mathscr{P}_{\theta}(\mathbb{R}^n)\times\mathbb{R}^{m}\rightarrow\mathbb{R}^{m}\otimes\mathbb{R}^{m},
		\end{align*}
		with
		$$\mathscr{P}_{\theta}(\mathbb{R}^{n}):=\Big\{\mu\in\mathbb{P}(\mathbb{R}^{n}):\mu(|\cdot|^{\theta}):=
		\int_{\mathbb{R}^{n}}|x|^{\theta}\mu(dx)<\infty\Big\},\quad \theta\in[2,\infty),$$
		where $\mathbb{P}$ is the set of probability measure on $(\mathbb{R}^{n},\mathcal{B}(\mathbb{R}^{n}))$. The space $\mathscr{P}_{\theta}(\mathbb{R}^{n})$ is a Polish space under the $L^{\theta}$-Wasserstein distance ($\theta\geq2$)
		$$\mathbb{W}_{\theta}(\mu_{1},\mu_{2}):=\inf_{\pi\in\mathscr{C}(\mu_{1},\mu_{2})}\Big(\int_{\mathbb{R}^{n}
			\times\mathbb{R}^{n}}|x-y|^{\theta}\pi(dx,dy)\Big)^{\frac{1}{\theta}},\quad \mu_{1},\mu_{2}\in\mathscr{P}_{\theta}(\mathbb{R}^{n}),$$
		where $\mathscr{C}(\mu_{1},\mu_{2})$ is the set of probability measures on $\mathbb{R}^{n}\times\mathbb{R}^{n}$ with marginals $\mu_{1}$ and $\mu_{2}$, respectively.
		The coefficients satisfy the following conditions.
		
		\textbf{$\mathbf{(H1)}$}
		There exists a non-decreasing function $K(t),$ $K(0)=1$ such that for any $t,t_{i}\in[0,T]$, $p>0$,
		$x_{i}\in\mathbb{R}^{n}$, $y_{i}\in \mathbb{R}^{m}$, $\mu_{i}\in\mathscr{P}_{\theta}(\mathbb{R}^{n})$,
		$\nu_{i}\in\mathscr{P}_{\theta}(\mathbb{R}^{m})$, $i=1,2,$
		\begin{equation*}\label{sec2-eq2.1}
			|b(t_{1},x_{1},\mu_{1},y_{1})-b(t_{2},x_{2},\mu_{2},y_{2})|^{p}\leq K(|t_{1}-t_{2}|^{p})[\kappa(|x_{1}-x_{2}|^{p}
			+|y_{1}-y_{2}|^{p}+\mathbb{W}_{\theta}(\mu_{1},\mu_{2})^{p})],
		\end{equation*}
		\begin{equation*}\label{sec2-eq2.2}
			\|\sigma(t,\mu_{1})-\sigma(t,\mu_{2})\|^{p}\leq K(t^{p})\kappa(\mathbb{W}_{\theta}(\mu_{1},\mu_{2})^{p}),
		\end{equation*}
		\begin{equation*}\label{sec2-eq2.3}
			|f(t_{1},x_{1},\mu_{1},y_{1})-f(t_{2},x_{2},\mu_{2},y_{2})|^{p}\leq K(|t_{1}-t_{2}|^{p})[\kappa(|x_{1}-x_{2}|^{p}
			+|y_{1}-y_{2}|^{p}+\mathbb{W}_{\theta}(\mu_{1},\mu_{2})^{p})],
		\end{equation*}
		\begin{equation*}\label{sec2-eq2.4}
			\|g(t_{1},x_{1},\mu_{1},y_{1})-g(t_{2},x_{2},\mu_{2},y_{2})\|^{p}\leq K(|t_{1}-t_{2}|^{p})[\kappa(|x_{1}-x_{2}|^{p}
			+|y_{1}-y_{2}|^{p}+\mathbb{W}_{\theta}(\mu_{1},\mu_{2})^{p})],
		\end{equation*}
		and
		\begin{equation*}\label{sec2-eq2.5}
			|b(t,0,\delta_{0},0)|^{p}+\|\sigma(t,\delta_{0})\|^{p}+|f(t,0,\delta_{0},0)|^{p}+\|g(t,0,\delta_{0},0)\|^{p}\leq K(t^{p}),
		\end{equation*}
		where $\kappa:\mathbb{R}^{+}\rightarrow\mathbb{R}^{+}$ is continuous and non-decreasing concave function with $\kappa(0)=0$, $\kappa(u)>0$, for every $u>0$ such that $\int_{0^{+}}\frac{1}{\kappa(u)}du=+\infty$.
		
		\textbf{$\mathbf{(H2)}$}
		There exist constants $\beta_{i}>0, i=1,2$, such that the following hold
		\begin{equation*}\label{sec1-H2-1}
			\begin{aligned}
				&2\langle y_{1}-y_{2},f(t_{1},x_{1},\mu_{1},y_{1})-f(t_{2},x_{2},\mu_{2},y_{2})\rangle
				+\|g(t_{1},x_{1},\mu_{1},y_{1})-g(t_{2},x_{2},\mu_{2},y_{2})\|^{2}\\
				&\leq-\beta_{1}|y_{1}-y_{2}|^{2}+K(|t_{1}-t_{2}|^{2})\kappa(|x_{1}-x_{2}|^{2}+\mathbb{W}_{2}(\mu_{1},\mu_{2})^{2}),
			\end{aligned}
		\end{equation*}
		and
		\begin{equation*}\label{sec1-H2-1}
			2\langle y,f(t,x,\mu,y)\rangle+\|g(t,x,\mu,y)\|^{2}\leq-\beta_{2}|y|^{2}+C_{T}(1+|x|^{2}+\mu(|\cdot|^{2})).
		\end{equation*}
		
		\begin{example}
			We can give a few concrete examples of the function   $\kappa(\cdot)$. Let $K>0$, and let $\gamma\in(0,1)$ be sufficiently small. Define\\
			$\kappa_1(u)=Ku,u\geq 0.$\\
			$\kappa_2(u)=\left\{
			\begin{array}{ll}
				u\log(u^{-1}), & {0\leq u\leq\gamma;} \\
				\gamma\log(\gamma^{-1})+\kappa{'}_2(\gamma-)(u-\gamma), & {u>\gamma.}
			\end{array}
			\right.$\\
			$\kappa_3(u)=\left\{
			\begin{array}{ll}
				u\log(u^{-1})\log\log(u^{-1}), & {0\leq u\leq\gamma;} \\
				\gamma\log(\gamma^{-1})\log\log(\gamma^{-1})+\kappa{'}_3(\gamma-)(u-\gamma), & {u>\gamma,}
			\end{array}
			\right.$\\
			where $\kappa{'}$ denotes the derivative of the function $\kappa$. They are all concave nondecreasing
			functions satisfying $\int_{0^+} {\frac{du}{\kappa_i(u)}}=\infty$. Furthermore, we observed that the
			Lipschitz condition is a special case of our proposed condition.
		\end{example}

		In order to prove the main results, we introduce the following useful le{mma, which presents a maximal inequlity for $\int_{0}^{t}\sigma(s,\mu_{s})dB_{s}^{H}$.
			\begin{lemma}(\cite{Fan})\label{sec2-lemma2.1}
				Suppose that $\sigma$ satisfies $\mathbf{(H1)}$ and $\mu\in C([0,T];\mathscr{P}_{p}(\mathbb{R}^{d}))$ with $p\geq \theta$ and $p>1/H$. Then there is a constant $C_{T, p, H}>0$ such that
				$$\mathbb{E}\Big(\sup_{t\in[0,T]}|\int_{0}^{t}\sigma(s,\mu_{s})dB_{s}^{H}|^{p}\Big)\leq
				C_{T, p, H}\int_{0}^{T}\|\sigma(s,\mu_{s})\|^{p}ds.$$
			\end{lemma}
			
			

			
			
			\section{Main result}\label{sec3}
			
			In order to get the LDP for the Eq. (\ref{sec1-eq1.2}), we first  recall some definitions of the theory of LDP and Laplace principle and their relations. Let $\{X^{\delta}\}_{\delta>0}$ denote a family of random variables defined on a complete probability space $(\Omega,\mathscr{F},\{\mathscr{F}_{t}\}_{t\geq0},\mathbb{P})$ taking values in a Polish space $\mathscr{E}$.
			
			\begin{definition}\label{def3.1}(Rate function) 
				A function $I:\mathscr{E}\rightarrow[0,+\infty)$ is called a rate function if $I$ is lower semicontinous. Moreover, a rate function $I$ is called a good rate function if for each constant $K<\infty,$ the level set $\{x\in\mathscr{E}:I(x)\leq K\}$ is a compact subset of $\mathscr{E}.$
			\end{definition}
			
			\begin{definition}\label{def3.2}(Large deviation principle) 
				The random variable family $\{X^{\delta}\}_{\delta>0}$ is said to satisfy the LDP on $\mathscr{E}$ with rate function $I$ if the following two conditions hold:
				\begin{itemize}
					\item [(i)] (LDP lower bound) For any open set $G\subset\mathscr{E}$,
					$$\liminf_{\delta\rightarrow0}\delta\log\mathbb{P}(X^{\delta}\in G)\geq-\inf_{x\in G}I(x).$$
					\item [(ii)](LDP upper bound) For any closed set $F\subset\mathscr{E}$,
					$$\limsup_{\delta\rightarrow0}\delta\log\mathbb{P}(X^{\delta}\in F)\leq-\inf_{x\in F}I(x).$$
				\end{itemize}
			\end{definition}
			
			Now, we recall the  Laplace principle, and then introduce the powerful weak convergence approach.

			\begin{definition}\label{def3.3}(Laplace principle) 
				The sequence $\{X^{\delta}\}_{\delta>0}$ is said to be satisfied the Laplace principle upper bound (respectively, lower bound) on $\mathscr{E}$ with a rate function $I$ if for each bounded continuous real-valued function $\phi$ defined on $C_{b}(\mathscr{E})$,
				$$\limsup_{\delta\rightarrow0}-\delta\log\mathbb{E}\Big\{\exp[-\frac{1}{\delta}\phi(X^{\delta})]\Big\}\leq \inf_{x\in\mathscr{E}}\Big(\phi(x)+I(x)\Big)$$ 	
				$\Big($respectively,
				$$\liminf_{\delta\rightarrow0}-\delta\log\mathbb{E}\Big\{\exp[-\frac{1}{\delta}\phi(X^{\delta})]\Big\}\geq \inf_{x\in\mathscr{E}}\Big(\phi(x)+I(x)\Big)\Big).$$ 	
			\end{definition}
			
			It is known that if $\mathscr{E}$ is a Polish space and $I$ is a good rate function, then the LDP and
			Laplace principle are equivalent from the Varadhan's lemma \cite{Va},  Budhiraja and  Dupuis \cite{Budhiraja1}.
			
			Next, we give the sufficient condition for a fractional version Laplace principle (see Fan et al. \cite{Fan}).
			\begin{hypothesis}\label{hypo3.4}
				There exists a measurable map $\mathcal{G}^{0}:I_{0+}^{H+\frac{1}{2}}(L^{2}([0,T],\mathbb{R}^{n+m}))\rightarrow\mathscr{E}$ for which the following two conditions hold:
				\begin{itemize}
					\item [(i)]
					Let $\{h^{\delta}\}_{\delta>0}\subset S_{M}$ for any $M\in (0, \infty)$ such that $h^{\delta}$ converges to element $h$ in $S_{M}$ as $\delta\rightarrow0$, then $\mathcal{G}^{0}(\int_{0}^{t}\dot{h}^{\delta}_{s}ds)$ converges to $\mathcal{G}^{0}(\int_{0}^{t}\dot{h}_{s}ds)$ in $\mathscr{E}$.
					
					\item [(ii)] Let $\{h^{\delta}\}_{\delta>0}\subset\mathcal{A}_{M}$ for any $M\in (0, \infty)$. For any $\epsilon_{0}>0$, we have	$$\lim_{\delta\rightarrow0}\mathbb{P}\Big(d(\mathcal{G}^{\delta}(\delta^{H}B^{H}_{t}
					+\int_{0}^{t}\dot{h}^{\delta}_{s}ds),\mathcal{G}^{0}(\int_{0}^{t}\dot{h}_{s}ds))>\epsilon_{0}\Big)=0,$$
					where $d(\cdot,\cdot)$ denotes the metric in $\mathscr{E}.$
				\end{itemize}
			\end{hypothesis}
			
			\begin{lemma}\label{sec3-lemma3.5}(\cite{Bourguin})
				If $X^{\delta}=\mathcal{G}^{\delta}(\delta^{H}B^{H}_{\cdot})$ and Hypothesis \ref{hypo3.4} holds, then the family $\{X^{\delta}\}_{\delta>0}$
				satisfies the Laplace principle (hence the LDP) in $\mathscr{E}$ with the good rate function $I$ given by
				\begin{equation}\label{sec3-eq3.0}	I(f)=\inf_{\{h\in\mathcal{H}:f=\mathcal{G}^{0}(\int_{0}^{\cdot}\dot{h}_{t}dt)\}}\Big\{\frac{1}{2}\int_{0}^{\cdot}|\dot{h}_{t}|^2dt\Big\},\qquad f\in\mathscr{E},
				\end{equation}
				where infimum over an empty set is taken as $+\infty.$
			\end{lemma}

			Our main result of LDP is as follows.
			\begin{theorem}\label{sec3-th3.5}
				Suppose that assumptions $\mathbf{(H1)}$, $\mathbf{(H2)}$ and the following conditions hold. 
				\begin{itemize}
					\item [(i)] There exists a constant $C_{T}>0$ such that for any $x\in\mathbb{R}^{n}$, $\mu\in\mathscr{P}_{2}(\mathbb{R}^{n})$,
					\begin{equation}\label{sec3-eq3.1}
						\sup_{y\in\mathbb{R}^{m}}\|g(t,x,\mu,y)\|\leq C_{T}\Big(1+|x|+[\mu(|\cdot|^{2})]^{\frac{1}{2}}\Big).
					\end{equation}
					\item [(i)] The scale parameters $\delta$ and $\epsilon$ satisfy
					\begin{equation}\label{sec3-eq3.2}
						\lim_{\delta\rightarrow0}\frac{\epsilon}{\delta}=0.
					\end{equation}
				\end{itemize}
				Then the solution $\{X^{\delta,\epsilon}\}_{\delta>0}$ of Eq. \eqref{sec1-eq1.2} satisfies the LDP in $C([0,T];\mathbb{R}^{n})$ with the good rate function $I$ given by \eqref{sec3-eq3.0}
				and map $\mathcal{G}^{0}$ will be defined in (\ref{sec1-eq1.6}).
			\end{theorem}

			\section{Proof of large deviations}\label{sec4}
			In order to carry out the complete proof of the LDP, we need to formulate the correct form of the skeleton equation. Actually, it is easy to find as the parameter $\delta\rightarrow0$ in Eq. (\ref{sec1-eq1.2}) (hence $\epsilon\rightarrow0$ also), the drift term is averaged and the noise term vanishes, then we can get the following differential equation
			
			\begin{equation}\label{sec1-eq1.4}
				\left\{
				\begin{aligned}
					&d\bar{X}_{t}^{0}=\bar{b}(t,\bar{X}_{t}^{0},\mathscr{L}_{\bar{X}_{t}^{0}})dt,\\
					& \bar{X}_{0}^{0}=x \in\mathbb{R}^{n},
				\end{aligned}
				\right.
			\end{equation}
			where $\mathscr{L}_{\bar{X}_{t}^{0}}=\delta_{\bar{X}_{t}^{0}}$ is the Dirac measure of $\bar{X}_{t}^{0}$.
			Then for $h$ belonging to the Cameron-Martin space $\mathcal{H}$ defined in Section \ref{sec2}, we can define the following skeleton equation with respect to the slow equation in Eq. (\ref{sec1-eq1.2})
			
			\begin{equation}\label{sec1-eq1.5}
				\left\{
				\begin{aligned}
					&d\bar{X}_{t}^{h}=\bar{b}(t,\bar{X}_{t}^{h},\mathscr{L}_{\bar{X}_{t}^{0}})dt+\sigma(t,\mathscr{L}_{\bar{X}_{t}^{0}})\dot{u}_{t}dt,\\
					&\bar{X}_{0}^{h}=x\in\mathbb{R}^{n},
				\end{aligned}
				\right.
			\end{equation}
			where $\bar{X}^{0}$ is the solution of Eq. (\ref{sec1-eq1.4}) and $\bar{b}$ is defined by  (\ref{sec1-eq1.31}), $\dot{u}$ is defined by Lemma \ref{sec2-lemma2.0}. Furthermore, we assume that $\sigma$ is H\"{o}lder continous of order belonging to $(1-H,1]$ with respect to the time variable.  Then, $\int_{0}^{t}\sigma(t,\mathscr{L}_{\bar{X}_{t}^{0}})\dot{u}_{t}dt$ in Eq. \eqref{sec1-eq1.5} is well-defined as a Riemann-Stieltjes integral. Thus, we can define a map
			$ \mathcal{G}^{0}:I_{0+}^{H+\frac{1}{2}}(L^{2}([0,T];\mathbb{R}^{n+m}))\rightarrow C([0,T];\mathbb{R}^{n})$
			\begin{equation}\label{sec1-eq1.6}
				\mathcal{G}^{0}\Big(\int_{0}^{t}\dot{u}_{s}ds\Big)=\bar{X}^{h}_{t},
			\end{equation}
			where the definition of space $I_{0+}^{H+\frac{1}{2}}(L^{2}([0,T];\mathbb{R}^{n+m}))$ has been introduced in Section \ref{sec2}.
			
			Since the existence and uniqueness for solutions of Eq. \eqref{sec1-eq1.2} has been established in Shen et al. \cite{Shen2}. Thus,
			according to the classical Yamada-Watanabe theorem (see  Hong et al. \cite{Hong}), there exists a measurable map $\mathcal{G}_{\mu^{\delta,\epsilon}}:C([0,T];\mathbb{R}^{n+m})\rightarrow C([0,T];\mathbb{R}^{n})$ such that we have the representation
			$$X_{t}^{\delta,\epsilon}=\mathcal{G}_{\mu^{\delta,\epsilon}}(\delta^{H}B^{H}_{t}).$$
			For simplicity of notation, we denote $\mathcal{G}^{\delta}=\mathcal{G}_{\mu^{\delta,\epsilon}}.$ Then for any $h^{\delta}\in \mathcal{A}_{M}$, we define
			$$X_{t}^{\delta,\epsilon,h^{\delta}}=\mathcal{G}^{\delta}(\delta^{H}B^{H}_{t}+\int_{0}^{t}\dot{h}^{\delta}_{s}ds),$$
			then $X_{t}^{\delta,\epsilon,h^{\delta}}$ satisfies the following  stochastic control equations
			\begin{equation}\label{sec1-eq1.7}
				\left\{
				\begin{aligned}
					&dX_{t}^{\delta,\epsilon,h^{\delta}}=b(t,X_{t}^{\delta,\epsilon,h^{\delta}},\mathscr{L}_{X_{t}^{\delta,\epsilon}},Y_{t}^{\delta,\epsilon,h^{\delta}})dt+\sigma(t,\mathscr{L}_{X_{t}^{\delta,\epsilon}})\dot{u}^{\delta}_{t}dt	+\delta^{H}\sigma(t,\mathscr{L}_{X_{t}^{\delta,\epsilon}})dB_{t}^{H},\\
					&dY_{t}^{\delta,\epsilon,h^{\delta}}=\frac{1}{\epsilon}f(t,X_{t}^{\delta,\epsilon,h^{\delta}},\mathscr{L}_{X_{t}^{\delta,\epsilon}},Y_{t}^{\delta,\epsilon,h^{\delta}})dt+\frac{1}{\sqrt{\delta\epsilon}}g(t,X_{t}^{\delta,\epsilon,h^{\delta}},\mathscr{L}_{X_{t}^{\delta,\epsilon}},Y_{t}^{\delta,\epsilon,h^{\delta}})\dot{v}_{t}^{\delta}dt\\
					&\qquad\qquad+\frac{1}{\sqrt{\epsilon}}g(t,X_{t}^{\delta,\epsilon,h^{\delta}},\mathscr{L}_{X_{t}^{\delta,\epsilon}},Y_{t}^{\delta,\epsilon,h^{\delta}})dW_{t},\\
					&X_{0}^{\delta,\epsilon,h^{\delta}}=x\in \mathbb{R}^{n},\qquad Y_{0}^{\delta,\epsilon,h^{\delta}}=y\in \mathbb{R}^{m}.
				\end{aligned}
				\right.
			\end{equation}

			\subsection{Some priori  estimates}\label{sec4.1}
			In order to prove the main result, we just need to verify the weak convergence criterions $(i)$ and $(ii)$ in  Hypothesis \ref{hypo3.4}, which will be presented in Propositions \ref{sec4-prop4.2} and \ref{sec4-prop4.3}, respectively. We  first give some necessary estimates.
			
			\begin{lemma}\label{sec4-lemma4.1}
				Suppose that assumptions $\mathbf{(H1)}$, $\mathbf{(H2)}$ and the condition \eqref{sec3-eq3.2} hold.  For any $x\in\mathbb{R}^{n}$ and $h\in\mathcal{H}$, there exists a unique solution to Eq. \eqref{sec1-eq1.5} satisfying
				$$\sup_{h\in S_{M}}\Big\{\sup_{t\in[0,T]}|\bar{X}_{t}^{h}|^{2}\Big\}\leq C_{T,H,M,|x|},$$
				where $C_{T,H,M,|x|}$ is a positive constant depending on $T,H,M,|x|$.
			\end{lemma}
			
			\begin{proof}
				It is easy to see that $\bar{b}(t,\bar{X}_{t}^{h},\mathscr{L}_{\bar{X}_{t}^{0}})$ and $\sigma(t,\mathscr{L}_{\bar{X}_{t}^{0}})$ independent of the measure satisfy $\mathbf{(H1)}$, which ensure the Eq. \eqref{sec1-eq1.5} has a unique solution (see Nualart and R\u{a}\c{s}canu \cite{Nualart0}). 
				It follows from  the change-of-variables formula  (Z\"{a}hle \cite{Zahle}), we have
				\begin{equation*}
					\begin{aligned}
						|\bar{X}_{t}^{h}|^{2}&=|x|^{2}+2\int_{0}^{t}\left\langle \bar{X}_{s}^{h},\bar{b}(s,\bar{X}_{s}^{h},\mathscr{L}_{\bar{X}_{s}^{0}})\right\rangle ds+2\int_{0}^{t} \left\langle \bar{X}_{s}^{h},\sigma(s,\mathscr{L}_{\bar{X}_{s}^{0}})\dot{u}_{s}\right\rangle ds\\
						&=:|x|^{2}+J_{1}(t)+J_{2}(t).
					\end{aligned}
				\end{equation*}
				For the term $J_{1}(t),$ we can obtain
				\begin{equation*}
					J_{1}(t)\leq2\int_{0}^{t}|\bar{X}_{s}^{h}|\cdot|\bar{b}(s,\bar{X}_{s}^{h},\mathscr{L}_{\bar{X}_{s}^{0}})-\bar{b}(s,0,\delta_{0})|ds+2\int_{0}^{t}|\bar{X}_{s}^{h}|\cdot|\bar{b}(s,0,\delta_{0})|ds.
				\end{equation*}
				According to  Shen et al. \cite{Shen2}, $\bar{b}(t,x,\mu)$   satisfies
				\begin{equation*}\label{sec4-eq4.1}
					|\bar{b}(t_{1},x_{1},\mu_{1})-\bar{b}(t_{2},x_{2},\mu_{2})|^{2}\leq K(|t_{1}-t_{2}|^{2})[\kappa(|x_{1}-x_{2}|^{2}+\mathbb{W}_{2}(\mu_{1},\mu_{2})^{2})],
				\end{equation*}
				where the non-decreasing function $K(t)$ and function $\kappa(\cdot)$	are the same defined in assumptions $\mathbf{(H1)}$.
				Given that $\kappa(\cdot)$ is concave and increasing, there must exists a positive number $a$ such that
				$$\kappa(u) \leq a(1+u).$$
				Therefore, we  have
				\begin{equation}\label{sec4-eq4.2}
					\begin{aligned}
						J_{1}(t)&\leq 2\int_{0}^{t}(|\bar{X}_{s}^{h}|\cdot|\bar{b}(s,\bar{X}_{s}^{h},\mathscr{L}_{\bar{X}_{s}^{0}})-\bar{b}(s,0,\delta_{0})|)ds+2\int_{0}^{t}K(s)|\bar{X}_{s}^{h}|ds \\
						&\leq\int_{0}^{t}(|\bar{X}_{s}^{h}|^{2}+|\bar{b}(s,\bar{X}_{s}^{h},\mathscr{L}_{\bar{X}_{s}^{0}})-\bar{b}(s,0,\delta_{0})|^{2})ds+\int_{0}^{t}(K^{2}(s)+|\bar{X}_{s}^{h}|^{2})ds\\
						&\leq\int_{0}^{t}(|\bar{X}_{s}^{h}|^{2}+\kappa(|\bar{X}_{s}^{h}|^{2}+\mathbb{W}_{2}(\mathscr{L}_{\bar{X}_{s}^{0}},\delta_{0})^{2})ds+\int_{0}^{t}(K^{2}(s)+|\bar{X}_{s}^{h}|^{2})ds\\ &\leq\int_{0}^{t}(|\bar{X}_{s}^{h}|^{2}+a(1+|\bar{X}_{s}^{h}|^{2}+|\bar{X}_{s}^{0}|^{2}))ds+\int_{0}^{t}(K(s^{2})+|\bar{X}_{s}^{h}|^{2})ds\\
						&\leq{(2+a)\int_{0}^{t}|\bar{X}_{s}^{h}|^{2}ds+a\int_{0}^{t}|\bar{X}_{s}^{0}|^{2}ds+(K(t^{2})+a)t.}
					\end{aligned}
				\end{equation}
				For the term $J_{2}(t)$, using Lemma \ref{sec2-lemma2.0} and the isometry between $L^{2}([0,T];\mathbb{R}^{n})$ and $\mathcal{H}$, we have
				\begin{equation*}\label{sec4-eq4.3}
					\begin{aligned}
						J_{2}(t)&=2\int_{0}^{t}\left\langle \bar{X}_{s}^{h},\sigma(s,\mathscr{L}_{\bar{X}_{s}^{0}})\dot{u}_{s}\right\rangle ds 		\\
						&\leq \int_{0}^{t} \|\sigma(s,\mathscr{L}_{\bar{X}_{s}^{0}})\|^{2}|\dot{u}_{s}|^{2}ds+\int_{0}^{t}|\bar{X}_{s}^{h}|^{2}ds\\
						&= \int_{0}^{t} \|\sigma(s,\mathscr{L}_{\bar{X}_{s}^{0}})\|^{2}|\dot{K}_{H}\hat{u}_{s}|^{2}ds+\int_{0}^{t}|\bar{X}_{s}^{h}|^{2}ds\\
						&\leq C_{H}\int_{0}^{t} \|\sigma(s,\mathscr{L}_{\bar{X}_{s}^{0}})\|^{2}|\hat{u}_{s}|^{2}ds+\int_{0}^{t}|\bar{X}_{s}^{h}|^{2}ds\\
						&\leq C_{H} \|\sigma(\cdot,\mathscr{L}_{\bar{X}_{\cdot}^{0}})\|_{\mathcal{H}}^{2}\cdot\|u_{\cdot}\|^{2}_{\mathcal{H}_{H}}+\int_{0}^{t}|\bar{X}_{s}^{h}|^{2}ds\\
						&\leq C_{H,T} \|\sigma(\cdot,\mathscr{L}_{\bar{X}_{\cdot}^{0}})\|_{L^{2}}^{2}\cdot\|u_{\cdot}\|^{2}_{\mathcal{H}_{H}}+\int_{0}^{t}|\bar{X}_{s}^{h}|^{2}ds.\
					\end{aligned}
				\end{equation*}
				
				Thus, it follows that
				\begin{equation}\label{sec4-eq4.4}
					\begin{aligned}
						J_{2}(t)
						&\leq \int_{0}^{t}|\bar{X}_{s}^{h}|^{2}ds+C_{H,T}K(t^{2}){(\kappa(|\bar{X}_{t}^{0}|)^{2}+1)}\cdot\|u_{t}\|^{2}_{\mathcal{H}_{H}}.
					\end{aligned}
				\end{equation}
				Therefore, combining the estimates (\ref{sec4-eq4.2}) with (\ref{sec4-eq4.4}), we   have
				\begin{equation*}\label{sec4-eq4.5}
					\begin{aligned}
						\sup_{t\in[0,T]}|\bar{X}_{t}^{h}|^{2}\leq& |x|^{2} +
						(3+a)\int_{0}^{T}(\sup_{r\in[0,s]}|\bar{X}_{r}^{h}|^{2})ds+{a(\sup_{t\in[0,T]}|}\bar{X}_{t}^{0}|^{2})+(K(T^{2})+a)T\\
						&+C_{H,T}K(T^{2})(a(1+\sup_{t\in[0,T]}|\bar{X}_{t}^{0}|^{2})+1)\cdot\|u_{T}\|^{2}_{\mathcal{H}_{H}}.\\
					\end{aligned}
				\end{equation*}
				
				Thus, for any $h\in S_{M}$ and using Gronwall's inequality and the boundness of $\sup_{t\in[0,T]}|\bar{X}_{t}^{0}|^{2}$ and $\|u_{\cdot}\|^{2}_{\mathcal{H}_{H}}$, we can obtain
				\begin{equation*}\label{sec4-eq4.6}
					\begin{aligned}
						\sup_{t\in[0,T]}|\bar{X}_{t}^{h}|^{2}\leq& \Big(|x|^{2}+a(\sup_{t\in[0,T]}|\bar{X}_{t}^{0}|^{2})+
						C_{H,T}K(T^{2})(a(1+\sup_{t\in[0,T]}|\bar{X}_{t}^{0}|^{2})+1)\cdot\|u_{T}\|^{2}_{\mathcal{H}_{H}}\\
						&+(K(T^{2})+a)T\Big)\exp \{T(3+a)\}\\
						\leq& C_{T,H,M,|x|}.
					\end{aligned}
				\end{equation*}
				This completes the proof.
			\end{proof}
			
			\begin{lemma}\label{sec4-lemma4.2}
				Under the assumptions in Theorem \ref{sec3-th3.5}, for any $\{h^{\delta}\}_{\delta>0}\subset\mathcal{A}_{M}$, there exists constant $C_{T,H,a,M, \beta_{2}}>0$ such that
				\begin{equation}\label{sec4-lemma4.2-1}
					\mathbb{E}[\sup_{t\in[0,T]}|X_{t}^{\delta,\epsilon,h^{\delta}}|^{2}]\leq C_{T,H,a,M ,\beta_{2}}(1+|x|^{2}+|y|^{2})
				\end{equation}
				and
				\begin{equation}\label{sec4-lemma4.2-2}
					\mathbb{E}[\int_{0}^{T}|Y_{t}^{\delta,\epsilon,h^{\delta}}|^{2}dt]\leq C_{T,H,a,M,\beta_{2}}(1+|x|^{2}+|y|^{2}).
				\end{equation}
			\end{lemma}
			
			\begin{proof}
				It comes from Eq. \eqref{sec1-eq1.7}, we have
				\begin{equation}\label{sec4-lemma4.2-3}
					\begin{aligned}
						\mathbb{E}\Big(\sup_{t\in[0,T]}|X_{t}^{\delta,\epsilon,h^{\delta}}|^{2}\Big)\leq&4|x|^{2}+C_{T}\mathbb{E}\Big(\sup_{t\in[0,T]}\int_{0}^{t}\kappa(1+|X_{s}^{\delta,\epsilon,h^{\delta}}|^{2}+|X_{s}^{\delta,\epsilon}|^{2}+|Y_{s}^{\delta,\epsilon,h^{\delta}}|^{2})ds\Big)\\
						&+C_{T,H}\|u_{T}^{\delta}\|^{2}_{\mathcal{H}_{H}}\mathbb{E}\Big(\int_{0}^{T}(1+\mathbb{W}_{2}(\mathscr{L}_{X_{s}^{\delta,\epsilon}},\delta_{0})^{2})ds \Big)\\
						&\quad+C_{T,H,\delta}\mathbb{E}\Big(\int_{0}^{T}(1+\mathbb{W}_{2}(\mathscr{L}_{X_{s}^{\delta,\epsilon}},\delta_{0})^{2})ds \Big)\\
						\leq&4|x|^{2}+C_{T}\mathbb{E}\Big(\sup_{t\in[0,T]}\int_{0}^{t}\Big[a(1+|X_{s}^{\delta,\epsilon,h^{\delta}}|^{2}+|X_{s}^{\delta,\epsilon}|^{2}+|Y_{s}^{\delta,\epsilon,h^{\delta}}|^{2})+1\Big]ds\Big)\\
						&+C_{T,H,M,\delta}(1+\mathbb{E}(\sup_{t\in[0,T]}|X_{t}^{\delta,\epsilon}|^{2}))\\
					\end{aligned}
				\end{equation}
			\begin{equation}	
				\begin{aligned}
				\leq&4|x|^{2}+C_{T,a}\mathbb{E}\Big(\sup_{t\in[0,T]}\int_{0}^{t}(|X_{s}^{\delta,\epsilon,h^{\delta}}|^{2}+|Y_{s}^{\delta,\epsilon,h^{\delta}}|^{2})ds\Big)\\
						&+C_{T,H,M,\delta,a}(1+\mathbb{E}(\sup_{t\in[0,T]}|X_{t}^{\delta,\epsilon}|^{2}))\\
						\leq&4|x|^{2}+C_{T,a}\int_{0}^{T}\mathbb{E}\Big(\sup_{r\in[0,s]}(|X_{r}^{\delta,\epsilon,h^{\delta}}|^{2}+|Y_{r}^{\delta,\epsilon,h^{\delta}}|^{2})\Big)ds\\
						&+C_{T,H,M,\delta,a}(1+\mathbb{E}(\sup_{t\in[0,T]}|X_{t}^{\delta,\epsilon}|^{2})).\\
					\end{aligned}
				\end{equation}
				Now, we first need to estimate the term $\mathbb{E}(\sup_{t\in[0,T]}|X_{t}^{\delta,\epsilon}|^{2}).$ Using H\"{o}lder inequality, assumptions $\mathbf{(H1)}$ and Lemma \ref{sec2-lemma2.1}, we can get
				\begin{equation*}
					\begin{aligned}
						\mathbb{E}\Big(\sup_{t\in[0,T]}|X_{t}^{\delta,\epsilon}|^{2}\Big)&=\mathbb{E}\Big(\sup_{t\in[0,T]}\Big|x+\int_{0}^{t}b(s,X_{s}^{\delta,\epsilon},\mathscr{L}_{X_{s}^{\delta,\epsilon}},Y_{s}^{\delta,\epsilon})ds+\delta^{H}\int_{0}^{t}\sigma(s,\mathscr{L}_{X_{s}^{\delta,\epsilon}})dB_{s}^{H}\Big|^{2}\Big)\\
						&\leq3|x|^{2}+T\mathbb{E}\Big(\sup_{t\in[0,T]}\int_{0}^{t}|b(s,X_{s}^{\delta,\epsilon},\mathscr{L}_{X_{s}^{\delta,\epsilon}},Y_{s}^{\delta,\epsilon})|^{2}\Big)ds\\
						&\quad+3\delta^{2H}\mathbb{E}\Big(\sup_{t\in[0,T]}|\int_{0}^{t}\sigma(s,\mathscr{L}_{X_{s}^{\delta,\epsilon}})dB_{s}^{H}|^{2}\Big)\\
						&\leq3|x|^{2}+C_{T}\mathbb{E}\Big(\sup_{t\in[0,T]}\int_{0}^{t}K(s^{2})(\kappa(|X_{s}^{\delta,\epsilon}|^{2}+|Y_{s}^{\delta,\epsilon}|^{2}+\mathbb{W}_{2}(\mathscr{L}_{X_{s}^{\delta,\epsilon}},\delta_{0})^{2})+1)ds\Big)\\
						&\quad+3\delta^{2H}\mathbb{E}\Big(\sup_{t\in[0,T]}|\int_{0}^{t}\sigma(s,\mathscr{L}_{X_{s}^{\delta,\epsilon}})dB_{s}^{H}|^{2}\Big)\\
						&\leq C_{T,|x|,\delta,H,a}+C_{T,H,a,\delta}\int_{0}^{T}\mathbb{E}(\sup_{r\in[0,s]}|X_{r}^{\delta,\epsilon}|^{2})ds+C_{T,a}\int_{0}^{T}\mathbb{E}(\sup_{r\in[0,s]}|Y_{r}^{\delta,\epsilon}|^{2})ds.
					\end{aligned}
				\end{equation*}	
				Next, we  get the estimate of $\int_{0}^{T}\mathbb{E}(\sup_{r\in[0,s]}|Y_{r}^{\delta,\epsilon}|^{2})ds.$ Applying It\^{o} formula and assumptions $\mathbf{(H2)}$, we can get
				\begin{equation*}
					\begin{aligned}
						\frac{d}{dr}\mathbb{E}\Big(\sup_{r\in[0,s]}|Y_{r}^{\delta,\epsilon}|^{2}\Big)&=\frac{2}{\epsilon}\mathbb{E}\Big(\sup_{r\in[0,s]}\langle f(r,X_{r}^{\delta,\epsilon},\mathscr{L}_{X_{r}^{\delta,\epsilon}}, Y_{r}^{\delta,\epsilon}),Y_{r}^{\delta,\epsilon}\rangle\Big)+\frac{1}{\epsilon}\mathbb{E}\Big(\sup_{r\in[0,s]}\|g(r,X_{r}^{\delta,\epsilon},\mathscr{L}_{X_{r}^{\delta,\epsilon}},Y_{r}^{\delta,\epsilon})\|^{2}\Big)\\
						&\leq -\frac{\beta_{2}}{\epsilon}\mathbb{E}(\sup_{r\in[0,s]}|Y_{r}^{\delta,\epsilon}|^{2})+\frac{C}{\epsilon}\Big(1+\mathbb{E}(\sup_{r\in[0,s]}|X_{r}^{\delta,\epsilon}|^{2})\Big).
					\end{aligned}
				\end{equation*}
				
				By the comparison theorem, we have
				
				\begin{equation*}
					\begin{aligned}
						\int_{0}^{T}\mathbb{E}\Big(\sup_{r\in[0,s]}|Y_{r}^{\delta,\epsilon}|^{2}\Big)ds&\leq|y|^{2}e^{-\frac{\beta_{2}T}{\epsilon}}+\frac{C}{\epsilon}\int_{0}^{T}e^{-\frac{\beta_{2}(T-s)}{\epsilon}}(\mathbb{E}(\sup_{r\in[0,s]}|X_{r}^{\delta,\epsilon}|^{2})+1)ds.\\
					\end{aligned}
				\end{equation*}
				Then,
				\begin{equation*}
					\begin{aligned}
						\mathbb{E}\Big(\sup_{t\in[0,T]}|X_{t}^{\delta,\epsilon}|^{2}\Big)\leq C_{T,|x|,|y|,a,\beta_{2},\epsilon}+C_{T,H,a,\delta}(1+\frac{1}{\epsilon})\int_{0}^{T}\mathbb{E}(\sup_{r\in[0,s]}|X_{r}^{\delta,\epsilon}|^{2})ds.
					\end{aligned}
				\end{equation*}
				Applying Gronwall's inequality, it is easy to get
				$$\mathbb{E}\Big(\sup_{t\in[0,T]}|X_{t}^{\delta,\epsilon}|^{2}\Big)\leq C_{T,H,|x|,|y|,a,\beta_{2},\delta,\epsilon}.$$
				Recall the equation of $Y_{t}^{\delta,\epsilon,h^{\delta}}$,
				\begin{equation*}
					\begin{aligned}
						dY_{t}^{\delta,\epsilon,h^{\delta}}
						=&\frac{1}{\epsilon}f(t,X_{t}^{\delta,\epsilon,h^{\delta}},\mathscr{L}_{X_{t}^{\delta,\epsilon}},Y_{t}^{\delta,\epsilon,h^{\delta}})dt+\frac{1}{\sqrt{\delta\epsilon}}g(t,X_{t}^{\delta,\epsilon,h^{\delta}},\mathscr{L}_{X_{t}^{\delta,\epsilon}},Y_{t}^{\delta,\epsilon,h^{\delta}})\dot{v}_{t}^{\delta}dt\\
						&+\frac{1}{\sqrt{\epsilon}}g(t,X_{t}^{\delta,\epsilon,h^{\delta}},\mathscr{L}_{X_{t}^{\delta,\epsilon}},Y_{t}^{\delta,\epsilon,h^{\delta}})dW_{t}.
					\end{aligned}
				\end{equation*}
				Then we can obtain
				\begin{equation*}\label{sec4-lemma4.2-4}
					\begin{aligned}
						\frac{d}{dt}\mathbb{E}\Big|Y_{t}^{\delta,\epsilon,h^{\delta}}\Big|^{2}\leq&
						\frac{2}{\epsilon}\mathbb{E}\Big(\langle f(t,X_{t}^{\delta,\epsilon,h^{\delta}},\mathscr{L}_{X_{t}^{\delta,\epsilon}},Y_{t}^{\delta,\epsilon,h^{\delta}}),Y_{t}^{\delta,\epsilon,h^{\delta}}\rangle\Big)+\frac{1}{\epsilon}\mathbb{E}\Big\|g(t,X_{t}^{\delta,\epsilon,h^{\delta}},\mathscr{L}_{X_{t}^{\delta,\epsilon}},Y_{t}^{\delta,\epsilon,h^{\delta}})\Big\|^{2}\\
						&+\frac{2}{\sqrt{\delta\epsilon}}\mathbb{E}\Big(\langle g (t,X_{t}^{\delta,\epsilon,h^{\delta}},\mathscr{L}_{X_{t}^{\delta,\epsilon}},Y_{t}^{\delta,\epsilon,h^{\delta}})\dot{v}_{t}^{\delta},Y_{t}^{\delta,\epsilon,h^{\delta}}\rangle\Big).
					\end{aligned}
				\end{equation*}
				By the assumptions in Theorem \ref{sec3-th3.5}, it can obtain that
				\begin{equation}\label{sec4-lemma4.2-5}
					\begin{aligned}
						&\frac{2}{\sqrt{\delta\epsilon}}\mathbb{E}\Big(\langle g (t,X_{t}^{\delta,\epsilon,h^{\delta}},\mathscr{L}_{X_{t}^{\delta,\epsilon}},Y_{t}^{\delta,\epsilon,h^{\delta}})\dot{v}_{t}^{\delta},Y_{t}^{\delta,\epsilon,h^{\delta}}\rangle\Big)\\
						&\leq
						\frac{C_{T,a}}{\sqrt{\delta\epsilon}}\mathbb{E}\Big((1+|X_{t}^{\delta,\epsilon,h^{\delta}}|+\sqrt{\mathscr{L}_{X_{t}^{\delta,\epsilon}}(|\cdot|^{2})})|\dot{v}_{t}^{\delta}|
						\cdot |Y_{t}^{\delta,\epsilon,h^{\delta}}|\Big)\\
						&\leq\frac{C_{T,a,\tilde{\beta}}}{\delta}\mathbb{E}\Big((1+|X_{t}^{\delta,\epsilon,h^{\delta}}|^{2}+\mathscr{L}_{X_{t}^{\delta,\epsilon}}(|\cdot|^{2}))|\dot{v}_{t}^{\delta}|^{2}\Big)
						+\frac{\tilde{\beta}}{\epsilon}\mathbb{E} |Y_{t}^{\delta,\epsilon,h^{\delta}}|^{2},
					\end{aligned}
				\end{equation}
				where $\tilde{\beta}\in(0,\beta_{2})$. By the assumption $\mathbf{(H2)}$, we have
				\begin{equation*}\label{sec4-lemma4.2-6}
					\begin{aligned}
						\frac{d}{dt}\mathbb{E}\Big|Y_{t}^{\delta,\epsilon,h^{\delta}}\Big|^{2}\leq&
						-\frac{k_{1}}{\epsilon}\mathbb{E}|Y_{t}^{\delta,\epsilon,h^{\delta}}|^{2}+\frac{C_{T,a}}{\epsilon}(1+\mathbb{E}|X_{t}^{\delta,\epsilon,h^{\delta}}|^{2}+\mathscr{L}_{X_{t}^{\delta,\epsilon}}(|\cdot|^{2}))\\
						&+\frac{C_{T,a,\tilde{\beta}}}{\delta}\mathbb{E}\Big((1+|X_{t}^{\delta,\epsilon,h^{\delta}}|^{2}+\mathscr{L}_{X_{t}^{\delta,\epsilon}}(|\cdot|^{2}))|\dot{v}_{t}^{\delta}|^{2}\Big),
					\end{aligned}
				\end{equation*}
				where $k_{1}:=\beta_{2}-\tilde{\beta}>0$.
				Applying the comparison theorem, it follows that
				\begin{equation}\label{sec4-lemma4.2-7}
					\begin{aligned}
						\mathbb{E}|Y_{t}^{\delta,\epsilon,h^{\delta}}|^{2}\leq&
						e^{-\frac{k_{1}}{\epsilon}t}|y|^{2}+\frac{C_{T,a}}{\epsilon}\int_{0}^{t}e^{-\frac{k_{1}}{\epsilon}(t-s)}(1+\mathbb{E}|X_{s}^{\delta,\epsilon,h^{\delta}}|^{2}+\mathscr{L}_{X_{s}^{\delta,\epsilon}}(|\cdot|^{2}))ds\\
						&+\frac{C_{T,a,\tilde{\beta}}}{\delta}\int_{0}^{t}e^{-\frac{k_{1}}{\epsilon}(t-s)}\mathbb{E}\Big((1+|X_{s}^{\delta,\epsilon,h^{\delta}}|^{2}+\mathscr{L}_{X_{s}^{\delta,\epsilon}}(|\cdot|^{2}))|\dot{v}_{s}^{\delta}|^{2}\Big)ds.
					\end{aligned}
				\end{equation}
				Then, integrating (\ref{sec4-lemma4.2-7}) with respect to $t$ from 0 to $T$,
				we can have
				\begin{equation}\label{sec4-lemma4.2-8}
					\begin{aligned}
						&\int_{0}^{T}\mathbb{E}|Y_{t}^{\delta,\epsilon,h^{\delta}}|^{2}dt\\
						&\leq C_{T}(1+|y|^{2})+C_{T,a,k_{1}}\int_{0}^{T}\mathbb{E}|X_{t}^{\delta,\epsilon,h^{\delta}}|^{2}
						dt+C_{T,a,k_{1}}\int_{0}^{T}\mathbb{E}|X_{t}^{\delta,\epsilon}|^{2}dt\\
						&\quad+C_{T,a,M,k_{1},\beta_{2}}(\frac{\epsilon}{\delta})\Big(1+\mathbb{E}(\sup_{t\in[0,T]}|X_{t}^{\delta,\epsilon,h^{\delta}}|^{2})+\mathbb{E}(\sup_{t\in[0,T]}|X_{t}^{\delta,\epsilon}|^{2})\Big).
					\end{aligned}
				\end{equation}
				Thus, combining (\ref{sec4-lemma4.2-3}) to (\ref{sec4-lemma4.2-8}), we can have
				\begin{equation*}\label{sec4-lemma4.2-9}
					\begin{aligned}
						\mathbb{E}\Big(\sup_{t\in[0,T]}|X_{t}^{\delta,\epsilon,h^{\delta}}|^{2}\Big)\leq& 4|x|^{2}+C_{T,a}\int_{0}^{T}\mathbb{E}\Big(\sup_{s\in[0,t]}|X_{s}^{\delta,\epsilon,h^{\delta}}|^{2}dt\Big)+C_{T,H,M,\delta,a}\Big(1+\mathbb{E}(\sup_{t\in[0,T]}|X_{t}^{\delta,\epsilon}|^{2})\Big)\\
						&+C_{T}(1+|y|^{2})+C_{T,a,k_{1}}\int_{0}^{T}\mathbb{E}(\sup_{s\in[0,t]}|X_{s}^{\delta,\epsilon,h^{\delta}}|^{2})dt\\
						&		+C_{T,a,k_{1}}\int_{0}^{T}\mathbb{E}(\sup_{s\in[0,t]}|X_{s}^{\delta,\epsilon}|^{2})dt\\
						&+C_{T,a,M,k_{1},\beta_{2}}(\frac{\epsilon}{\delta})\Big(1+\mathbb{E}(\sup_{t\in[0,T]}|X_{t}^{\delta,\epsilon,h^{\delta}}|^{2})+\sup_{t\in[0,T]}\mathbb{E}|X_{t}^{\delta,\epsilon}|^{2}\Big).
					\end{aligned}
				\end{equation*}
				
				Thus,
				with the condition $\lim_{\delta\rightarrow0}\frac{\epsilon}{\delta}=0$, we can choose $\frac{\epsilon}{\delta}\in(0,\frac{1}{2C_{T,a,M,k_{1},\beta_{2}}})$ such that
				\begin{equation*}\label{sec4-lemma4.2-10}
					\mathbb{E}\Big(\sup_{t\in[0,T]}|X_{t}^{\delta,\epsilon,h^{\delta}}|^{2}\Big)\leq C_{T,H,a,M, \beta_{2}}(1+|x|^{2}+|y|^{2})+C_{T,a}\int_{0}^{T}\mathbb{E}(\sup_{s\in[0,t]}|X_{s}^{\delta,\epsilon,h^{\delta}}|^{2})dt.
				\end{equation*}
				Using Gronwall's inequality, we can get the desired estimate (\ref{sec4-lemma4.2-1}). By the boundness of $\mathbb{E}(\sup_{t\in[0,T]}|X_{t}^{\delta,\epsilon}|^{2})$ and estimate (\ref{sec4-lemma4.2-1}), one can easily obtain (\ref{sec4-lemma4.2-2}).
				This completes the proof.
			\end{proof}

			\begin{lemma}\label{sec4-lemma4.3}
				Under the assumptions in Theorem \ref{sec3-th3.5}, {for any $h^{\delta}\in \mathcal{A}_{M}$, }there exists a constant $C_{T,H,a,M.|x|,|y|,\delta,\beta_{2}}$ such that for any $\Delta\in(0,1)$ and $t\in[0,T]$,
				\begin{equation}\label{sec4-lemma4.3.1}
					\mathbb{E}|X_{t}^{\delta,\epsilon}-X_{t(\Delta)}^{\delta,\epsilon}|^{2}\leq C_{T,H,a,M,|x|,|y|,\delta,\beta_{2}}(\Delta^{2}\vee\Delta^{2H})
				\end{equation}
				and
				\begin{equation}\label{sec4-lemma4.3-2}
					\mathbb{E}\Big[\int_{0}^{T}|X_{t}^{\delta,\epsilon,h^{\delta}}-X_{t(\Delta)}^{\delta,\epsilon,h^{\delta}}|^{2}dt\Big]\leq C_{T,H,a,M,|x|,|y|,\delta,\beta_{2}}(\Delta^{2}\vee\Delta^{2H}),
				\end{equation}
				where $t(\Delta):=[\frac{t}{\Delta}]\Delta$ and $[s]$ denotes the integer part of $s$.
			\end{lemma}
			
			\begin{proof}
				We first give the   estimate for  (\ref{sec4-lemma4.3.1}).
				\begin{equation*}
					\begin{aligned}
						&\mathbb{E}|X_{t}^{\delta,\epsilon}-X^{\delta,\epsilon}_{t(\Delta)}|^{2}\\
						=&\mathbb{E}\Big|\int_{t(\Delta)}^{t}b(s,X_{s}^{\delta,\epsilon},\mathscr{L}_{X_{s}^{\delta,\epsilon}},Y_{s}^{\delta,\epsilon})ds+\int_{t(\Delta)}^{t}\sigma(s,\mathscr{L}_{X_{s}^{\delta,\epsilon}})dB_{s}^{H}\Big|^{2}\\
						\leq& C\mathbb{E}\Big(\Big|\int_{t(\Delta)}^{t}b(s,X_{s}^{\delta,\epsilon},\mathscr{L}_{X_{s}^{\delta,\epsilon}},Y_{s}^{\delta,\epsilon})ds\Big|^{2}+\Big|\int_{t(\Delta)}^{t}\sigma(s,\mathscr{L}_{X_{s}^{\delta,\epsilon}})dB_{s}^{H}\Big|^{2}\Big)\\
						\leq& C_{T}\Delta\int_{t(\Delta)}^{t}K(s^{2})(\kappa(2\mathbb{E}|X_{s}^{\delta,\epsilon}|^{2}+\mathbb{E}|Y_{s}^{\delta,\epsilon}|^{2})+1)ds+C_{T}\Delta^{2H-1}\int_{t(\Delta)}^{t}K(s^{2})(\kappa(\mathbb{E}|X_{s}^{\delta,\epsilon}|^{2})+1)ds\\
						\leq&C_{T,H,a,M,|x|,|y|,\beta_{2}}(\Delta^{2}\vee \Delta^{2H}).
					\end{aligned}
				\end{equation*}
				
				Now, we give the estimate for (\ref{sec4-lemma4.3-2}). 
				\begin{equation}\label{sec4-lemma4.3-3}
					\begin{aligned}
						\mathbb{E}&\Big(\int_{0}^{T}|X_{t}^{\delta,\epsilon,h^{\delta}}-X_{t(\Delta)}^{\delta,\epsilon,h^{\delta}}|^{2}dt\Big)\\
						=&\mathbb{E}\Big(\int_{0}^{\Delta}|X_{t}^{\delta,\epsilon,h^{\delta}}-x|^{2}dt\Big)+\mathbb{E}\Big(\int_{\Delta}^{T}|X_{t}^{\delta,\epsilon,h^{\delta}}-X_{t(\Delta)}^{\delta,\epsilon,h^{\delta}}|^{2}dt\Big)\\
						\leq&C_{T,H,a,M,\beta_{2}}(1+|x|^{2}+|y|^{2})\Delta+2\mathbb{E}\Big(\int_{\Delta}^{T}|X_{t}^{\delta,\epsilon,h^{\delta}}-X_{t-\Delta}^{\delta,\epsilon,h^{\delta}}|^{2}dt\Big)\\
						&+2\mathbb{E}\Big(\int_{\Delta}^{T}|X_{t(\Delta)}^{\delta,\epsilon,h^{\delta}}-X_{t-\Delta}^{\delta,\epsilon,h^{\delta}}|^{2}dt\Big).\\
					\end{aligned}
				\end{equation}
				For the second term on the right-hand side of (\ref{sec4-lemma4.3-3}), by the basic inequality, we have
				\begin{equation*}\label{sec4-lemma4.3-4}
					\begin{aligned}
						\mathbb{E}&\Big(\int_{\Delta}^{T}|X_{t}^{\delta,\epsilon,h^{\delta}}-X_{t-\Delta}^{\delta,\epsilon,h^{\delta}}|^{2}dt\Big)\\
						\leq&3\mathbb{E}\Big|\int_{\Delta}^{T}\int_{t-\Delta}^{t}b(s,X_{s}^{\delta,\epsilon,h^{\delta}},\mathscr{L}_{X_{s}^{\delta,\epsilon}},Y_{s}^{\delta,\epsilon,h^{\delta}})dsdt\Big|^{2}\\
						&+3\mathbb{E}\Big|\int_{\Delta}^{T}\int_{t-\Delta}^{t}\sigma(s,\mathscr{L}_{X_{s}^{\delta,\epsilon}})\dot{u}_{s}^{\delta}dsdt\Big|^{2}\\
						&		+3\delta^{2H}\mathbb{E}\Big|\int_{\Delta}^{T}\int_{t-\Delta}^{t}\sigma(s,\mathscr{L}_{X_{s}^{\delta,\epsilon}})dB_{s}^{H}dt\Big|^{2}\\
						=&:O_{1}(t)+O_{2}(t)+O_{3}(t).
					\end{aligned}
				\end{equation*}
				Next, we will estimate the terms $O_{i}(t),i=1,2,3$, respectively.
				
				For the term $O_{1}(t)$, using assumptions $\mathbf{(H1)}$, it follows that
				\begin{equation*}\label{sec4-lemma4.3-5}
					\begin{aligned}
						O_{1}(t)&\leq C_{T}\Delta\int_{\Delta}^{T}\int_{t-\Delta}^{t}|b(s,X_{s}^{\delta,\epsilon,h^{\delta}},\mathscr{L}_{X_{s}^{\delta,\epsilon}},Y_{s}^{\delta,\epsilon,h^{\delta}})|^{2}dsdt\\
						&\leq C_{T}\Delta\int_{\Delta}^{T}\int_{t-\Delta}^{t}K(s^{2})\Big(a(1+|X_{s}^{\delta,\epsilon,h^{\delta}}|^{2}+|X_{s}^{\delta,\epsilon}|^{2}+|Y_{s}^{\delta,\epsilon,h^{\delta}}|^{2})+1\Big)dsdt\\
						&\leq C_{T,H,a,M,|x|,|y|, \beta_{2}}\Delta^{2}.
					\end{aligned}
				\end{equation*}
				
				For the terms $O_{2}(t)$ and $O_{3}(t)$, using Lemma \ref{sec2-lemma2.1}, it is easy to get
				\begin{equation*}\label{sec4-lemma4.3-6}
					\begin{aligned}
						O_{2}(t)+O_{3}(t)&
						\leq3\mathbb{E}\Big|\int_{\Delta}^{T}\int_{t-\Delta}^{t}\sigma(s,\mathscr{L}_{X_{s}^{\delta,\epsilon}})\dot{u}_{s}^{\delta}dsdt\Big|^{2}\\
						&		\quad+3T\delta^{2H}\Big|\int_{\Delta}^{T}\mathbb{E}\Big(\sup_{r\in[0,t]}\int_{r-\Delta}^{r}\sigma(s,\mathscr{L}_{X_{s}^{\delta,\epsilon}})dB_{s}^{H}\Big)^{2}dt\Big|\\
						&\leq
						3T\Delta\int_{\Delta}^{T}\int_{t-\Delta}^{t}\mathbb{E}\|\sigma(s,\mathscr{L}_{X_{s}^{\delta,\epsilon}})\|^{2}\cdot|\dot{u}_{s}^{\delta}|^{2}dsdt\\
						&\quad+3T\delta^{2H}\Delta^{2H-1}\int_{\Delta}^{T}\int_{t-\Delta}^{t}\|\sigma(s,\mathscr{L}_{X_{s}^{\delta,\epsilon}})\|^{2}dsdt\\
						&\leq C_{T,H,a,M,|x|,|y|,\delta,\beta_{2}}(\Delta^{2}\vee\Delta^{2H}).
					\end{aligned}
				\end{equation*}
				Hence, we  have
				\begin{equation}\label{sec4-lemma4.3-7}
					\mathbb{E}\Big(\int_{\Delta}^{T}	|X_{t}^{\delta,\epsilon,h^{\delta}}-X_{t-\Delta}^{\delta,\epsilon,h^{\delta}}|^{2}dt\Big)
					\leq C_{T,H,a,M,|x|,|y|,\delta,\beta_{2}}(\Delta^{2}\vee\Delta^{2H}).
				\end{equation}
				For the third term on the right-hand side of (\ref{sec4-lemma4.3-3}), by a similar argument  as (\ref{sec4-lemma4.3-7}), we can obtain
				\begin{equation}\label{sec4-lemma4.3-8}
					\begin{aligned}
						\mathbb{E}\Big(\int_{\Delta}^{T}	|X_{t(\Delta)}^{\delta,\epsilon,h^{\delta}}-X_{t-\Delta}^{\delta,\epsilon,h^{\delta}}|^{2}dt\Big)
						\leq &C_{T,H,a,M,|x|,|y|,\delta,\beta_{2}}(\Delta^{2}\vee\Delta^{2H}).
					\end{aligned}
				\end{equation}
				Thus, (\ref{sec4-lemma4.3-2}) can be derived from (\ref{sec4-lemma4.3-7}) and (\ref{sec4-lemma4.3-8}). This completes the proof.
			\end{proof}
			
			In the following discussion, in order to obtain the convergence of $X^{\delta,\epsilon,h^{\delta}}-\bar{X}^{h^{\delta}}$, we adopt the method of time discretization  techniques from Khasminskii \cite{Khasminskii}
			and introduce an auxiliary process $\bar{Y}_{t}^{\delta,\epsilon}\in\mathbb{R}^{m}$ with $\bar{Y}_{0}^{\delta,\epsilon}=Y_{0}^{\delta,\epsilon}=Y_{0}^{\delta,\epsilon,h^{\delta}}=y$ and for any $k\in\mathbb{N}$ and $t\in[k\Delta,\min\{(k+1)\Delta,T\}]$,
			\begin{equation*}\label{sec4-lemma4.4-0}
				\bar{Y}_{t}^{\delta,\epsilon}=\bar{Y}_{k\Delta}^{\delta,\epsilon}+\frac{1}{\epsilon}\int_{k\Delta}^{t}f(k\Delta,X_{k\Delta}^{\delta,\epsilon,h^{\delta}},\mathscr{L}_{X_{k\Delta}^{\delta,\epsilon}},\bar{Y}_{s}^{\delta,\epsilon})ds
				+	\frac{1}{\sqrt{\epsilon}}\int_{k\Delta}^{t}g(k\Delta, X_{k\Delta}^{\delta,\epsilon,h^{\delta}},\mathscr{L}_{X_{k\Delta}^{\delta,\epsilon}},\bar{Y}_{s}^{\delta,\epsilon})dW_s.
			\end{equation*}
			This can be rewritten as
			\begin{equation*}
				\bar{Y}_{t}^{\delta,\epsilon}=y+\frac{1}{\epsilon}\int_{0}^{t}f(s(\Delta),X_{s(\Delta)}^{\delta,\epsilon,h^{\delta}},\mathscr{L}_{X_{s(\Delta)}^{\delta,\epsilon}},\bar{Y}_{s}^{\delta,\epsilon})ds
				+	\frac{1}{\sqrt{\epsilon}}\int_{0}^{t}g(s(\Delta), X_{s(\Delta)}^{\delta,\epsilon,h^{\delta}},\mathscr{L}_{X_{s(\Delta)}^{\delta,\epsilon}},\bar{Y}_{s}^{\delta,\epsilon})dW_s.
			\end{equation*}
			
			Next, we aim to get the following error estimate between the process $Y^{\delta,\epsilon,h^{\delta}}$ and $\bar{Y}^{\delta,\epsilon}$.
			\begin{lemma}\label{sec4-lemma4.4}
				Under the assumptions in Theorem \ref{sec3-th3.5}, for any $M<\infty$, there exists a constant $C_{T,H,a,M,|x|,|y|,\beta_{1},\beta_{2}}>0$ such that
				\begin{equation*}\label{sec4-lemma4.3-1}
					\mathbb{E}\Big(\int_{0}^{T}|Y_{t}^{\delta,\epsilon,h^{\delta}}-\bar{Y}_{t}^{\delta,\epsilon}|^{2}dt\Big)\leq C_{T,H,a,M,|x|,|y|,\beta_{1},\beta_{2}}(\frac{\epsilon}{\delta}+\kappa(\Delta^{2}\vee\Delta^{2H})).
				\end{equation*}
			\end{lemma}
			\begin{proof}
				The boundness of $|\bar{Y}_{t}^{\delta,\epsilon}|$ is easy to obtain and we just omit the details here.  Actually, the method of proof in this lemma is similar to that of Hong et al. (\cite{Hong2}, Lemma 5.8). For completeness of the proof, we provide the main proof framework here.
				From the definitions of the process $Y_{t}^{\delta,\epsilon,h^{\delta}}$ and $\bar{Y}_{t}^{\delta,\epsilon}$, we have the following equations
				\begin{equation}\label{sec4-lemma4.4-2}
					\left\{
					\begin{aligned}
						&d(Y_{t}^{\delta,\epsilon,h^{\delta}}-\bar{Y}_{t}^{\delta,\epsilon})=\frac{1}{\epsilon}
						\Big[f(t,X_{t}^{\delta,\epsilon,h^{\delta}},\mathscr{L}_{X_{t}^{\delta,\epsilon}},Y_{t}^{\delta,\epsilon,h^{\delta}})-f(t(\Delta),X_{t(\Delta)}^{\delta,\epsilon,h^{\delta}},\mathscr{L}_{X_{t(\Delta)}^{\delta,\epsilon}},\bar{Y}_{t}^{\delta,\epsilon})\Big]dt\\
						&\qquad\qquad\qquad\qquad+\frac{1}{\sqrt{\epsilon}}\Big[g(t,X_{t}^{\delta,\epsilon,h^{\delta}},\mathscr{L}_{X_{t}^{\delta,\epsilon}},Y_{t}^{\delta,\epsilon,h^{\delta}})-g(t(\Delta),X_{t(\Delta)}^{\delta,\epsilon,h^{\delta}},\mathscr{L}_{X_{t(\Delta)}^{\delta,\epsilon}},\bar{Y}_{t}^{\delta,\epsilon})\Big]dW_{t}\\
						&\qquad\qquad\qquad\qquad+\frac{1}{\sqrt{\delta\epsilon}}g(t,X_{t}^{\delta,\epsilon,h^{\delta}},\mathscr{L}_{X_{t}^{\delta,\epsilon}},Y_{t}^{\delta,\epsilon,h^{\delta}})\dot{v}_{t}^{\delta}dt,\\
						&Y_{0}^{\delta,\epsilon,h^{\delta}}-\bar{Y}_{0}^{\delta,\epsilon}=0.
					\end{aligned}
					\right.
				\end{equation}
				It follows from It\^{o} formula that
				\begin{equation}\label{sec4-lemma4.4-3}
					\begin{aligned}
						&\frac{d}{dt}\mathbb{E}\Big|Y_{t}^{\delta,\epsilon,h^{\delta}}-\bar{Y}_{t}^{\delta,\epsilon}
						\Big|^{2}\\
						&=\frac{2}{\epsilon}\mathbb{E}\Big\langle f(t,X_{t}^{\delta,\epsilon,h^{\delta}},\mathscr{L}_{X_{t}^{\delta,\epsilon}},Y_{t}^{\delta,\epsilon,h^{\delta}})-f(t(\Delta),X_{t(\Delta)}^{\delta,\epsilon,h^{\delta}},\mathscr{L}_{X_{t(\Delta)}^{\delta,\epsilon}},\bar{Y}_{t}^{\delta,\epsilon}),Y_{t}^{\delta,\epsilon,h^{\delta}}-\bar{Y}_{t}^{\delta,\epsilon}\Big\rangle\\
						&\quad+\frac{2}{\sqrt{\delta\epsilon}}\mathbb{E}\Big\langle g(t,X_{t}^{\delta,\epsilon,h^{\delta}},\mathscr{L}_{X_{t}^{\delta,\epsilon}},Y_{t}^{\delta,\epsilon,h^{\delta}})\dot{v}_{t}^{\delta},Y_{t}^{\delta,\epsilon,h^{\delta}}-\bar{Y}_{t}^{\delta,\epsilon}\Big\rangle\\
						&\quad+\frac{1}{\epsilon}\mathbb{E}\Big\|g(t,X_{t}^{\delta,\epsilon,h^{\delta}},\mathscr{L}_{X_{t}^{\delta,\epsilon}},Y_{t}^{\delta,\epsilon,h^{\delta}})-g(t(\Delta),X_{t(\Delta)}^{\delta,\epsilon,h^{\delta}},\mathscr{L}_{X_{t(\Delta)}^{\delta,\epsilon}},\bar{Y}_{t}^{\delta,\epsilon})\Big\|^{2}\\
						&=:Q_{1}(t)+Q_{2}(t)+Q_{3}(t).
					\end{aligned}
				\end{equation}
				
				For the terms $Q_{1}(t)$ and $Q_{3}(t)$, we can get
				\begin{equation}\label{sec4-lemma4.4-4}
					Q_{1}(t)+Q_{3}(t)\leq -\frac{\beta_{1}}{\epsilon}\mathbb{E}|Y_{t}^{\delta,\epsilon,h^{\delta}}-\bar{Y}_{t}^{\delta,\epsilon}|^{2}+\frac{C_{T,a}}{\epsilon}\kappa\Big(\mathbb{E}|X_{t}^{\delta,\epsilon,h^{\delta}}-X_{t(\Delta)}^{\delta,\epsilon,h^{\delta}}|^{2}+\mathbb{E}|X_{t}^{\delta,\epsilon}-X_{t(\Delta)}^{\delta,\epsilon}|^{2}\Big).
				\end{equation}
				
				For the term $Q_{2}(t)$, we can get
				\begin{equation}\label{sec4-lemma4.4-5}
					Q_{2}(t)\leq -\frac{\tilde{\beta}}{\epsilon}\mathbb{E}|Y_{t}^{\delta,\epsilon,h^{\delta}}-\bar{Y}_{t}^{\delta,\epsilon}|^{2}+\frac{C_{T,a}}{\delta}\mathbb{E}(1+|X_{t}^{\delta,\epsilon,h^{\delta}}|^{2}+\mathscr{L}_{X_{t}^{\delta,\epsilon}}(|\cdot|^{2}))\cdot|\dot{v}_{t}^{\delta}|^{2},
				\end{equation}
				where $\tilde{\beta}\in(0,\beta_{2})$ is the same defined in Lemma \ref{sec4-lemma4.2}.
				
				Substituting (\ref{sec4-lemma4.4-4}) and (\ref{sec4-lemma4.4-5}) into (\ref{sec4-lemma4.4-3}), it follows that
				\begin{equation*}\label{sec4-lemma4.4-6}
					\begin{aligned}
						\frac{d}{dt}\mathbb{E}\Big|Y_{t}^{\delta,\epsilon,h^{\delta}}-\bar{Y}_{t}^{\delta,\epsilon}\Big|^{2}\leq&-\frac{k_{2}}{\epsilon}\mathbb{E}|Y_{t}^{\delta,\epsilon,h^{\delta}}-\bar{Y}_{t}^{\delta,\epsilon}|^{2}+\frac{C_{T,a}}{\epsilon}\kappa\Big(\mathbb{E}|X_{t}^{\delta,\epsilon,h^{\delta}}-X_{t(\Delta)}^{\delta,\epsilon,h^{\delta}}|^{2}
						+\mathbb{E}|X_{t}^{\delta,\epsilon}-X_{t(\Delta)}^{\delta,\epsilon}|^{2}\Big)\\
						&+\frac{C_{T,a}}{\delta}\mathbb{E}(1+|X_{t}^{\delta,\epsilon,h^{\delta}}|^{2}+\mathscr{L}_{X_{t}^{\delta,\epsilon}}(|\cdot|^{2}))|\dot{v}_{t}^{\delta}|^{2},
					\end{aligned}
				\end{equation*}
				where $k_{2}:=\beta_{1}+\tilde{\beta.}$
				Using the comparison theorem, we can have
				\begin{equation*}\label{sec4-lemma4.4-7}
					\begin{aligned}
						\mathbb{E}\Big|Y_{t}^{\delta,\epsilon,h^{\delta}}-\bar{Y}_{t}^{\delta,\epsilon}\Big|^{2}
						\leq&\frac{C_{T,a}}{\epsilon}\int_{0}^{t}e^{-\frac{k_{2}(t-s)}{\epsilon}}\kappa\Big(\mathbb{E}|X_{s}^{\delta,\epsilon,h^{\delta}}-X_{s(\Delta)}^{\delta,\epsilon,h^{\delta}}|^{2}+\mathbb{E}|X_{s}^{\delta,\epsilon}-X_{s(\Delta)}^{\delta,\epsilon}|^{2}\Big)ds\\
						&+\frac{C_{T,a}}{\delta}\int_{0}^{t}e^{-\frac{k_{2}(t-s)}{\epsilon}}\mathbb{E}(1+|X_{s}^{\delta,\epsilon,h^{\delta}}|^{2}+\mathscr{L}_{X_{s}^{\delta,\epsilon}}(|\cdot|^{2}))|\dot{v}_{s}^{\delta}|^{2}ds.
					\end{aligned}
				\end{equation*}
				Then we can obtain
				\begin{equation*}\label{sec4-lemma4.4-8}
					\begin{aligned}
						&\mathbb{E}\Big(\int_{0}^{T}|Y_{t}^{\delta,\epsilon,h^{\delta}}-\bar{Y}_{t}^{\delta,\epsilon}|^{2}dt\Big)\\
						&\leq\frac{C_{T,a}}{\epsilon}\Big[\int_{0}^{T}\kappa\Big(\mathbb{E}|X_{s}^{\delta,\epsilon,h^{\delta}}-X_{s(\Delta)}^{\delta,\epsilon,h^{\delta}}|^{2}+\mathbb{E}|X_{s}^{\delta,\epsilon}-X_{s(\Delta)}^{\delta,\epsilon}|^{2}\Big)(\int_{s}^{T}e^{-\frac{k_{2}(t-s)}{\epsilon}}dt)ds\Big]\\
						&\quad+\frac{C_{T,a}}{\delta}\Big(\int_{0}^{T}\mathbb{E}(1+|X_{s}^{\delta,\epsilon,h^{\delta}}|^{2}+|X_{s}^{\delta,\epsilon}|^{2})|\dot{v}_{s}^{\delta}|^{2}(\int_{s}^{T}e^{-\frac{k_{2}(t-s)}{\epsilon}}dt)ds\Big)\\
						&\leq\frac{C_{T,a}}{k_{2}}\int_{0}^{T}\kappa\Big(\mathbb{E}|X_{s}^{\delta,\epsilon,h^{\delta}}-X_{s(\Delta)}^{\delta,\epsilon,h^{\delta}}|^{2}+\mathbb{E}|X_{s}^{\delta,\epsilon}-X_{s(\Delta)}^{\delta,\epsilon}|^{2}\Big)ds\\
						&\quad+\frac{C_{T,a}}{k_{2}}(\frac{\epsilon}{\delta})\mathbb{E}\Big(\sup_{s\in[0,T]}(1+|X_{s}^{\delta,\epsilon,h^{\delta}}|^{2}+|X_{s}^{\delta,\epsilon}|^{2})\Big)\int_{0}^{T}|\dot{v}_{s}^{\delta}|^{2}ds\\
						\leq&C_{T,H,a,M,|x|,|y|,\beta_{1},\beta_{2}}(\frac{\epsilon}{\delta}+\kappa(\Delta^{2}\vee\Delta^{2H})).
					\end{aligned}
				\end{equation*}
				This completes the proof.
			\end{proof}
			
			\begin{lemma}\label{sec4-lemma4.5}
				Under the assumptions $\mathbf{(H1)}$ and the condition \eqref{sec3-eq3.2}, for any $T>0$ , $t\in[0,T]$, we have
				$$\lim_{\delta\rightarrow0}\mathbb{E}|X_{t}^{\delta,\epsilon}-\bar{X}_{t}^{0}|^{2}=0.$$
			\end{lemma}
			\begin{proof}
				The result follows by (Shen et al. \cite{Shen2}, Theorem 3.2) with slight modification and we just omit the details here.
			\end{proof}
			
			\subsection{The Proof of Main result }\label{sec4.2}
			In this subsection,  we aim to   verify the criterions $(i)$ and $(ii)$ in  Hypothesis \ref{hypo3.4}, these will be presented in Propositions \ref{sec4-prop4.2} and \ref{sec4-prop4.3}, respectively.
			\begin{proposition}\label{sec4-prop4.2}
				Under the assumptions $\mathbf{(H1)}$ and the condition (\ref{sec3-eq3.2}), let $\{h^{\delta}\}_{\delta>0}\subset S_{M}$ for any $M<\infty$ such that $h^{\delta}$ converges to element $h$ in $S_{M}$ as $\delta\rightarrow0$. Then $\mathcal{G}^{0}(\int_{0}^{t}\dot{h}^{\delta}_{s}ds)$ converges to $\mathcal{G}^{0}(\int_{0}^{t}\dot{h}_{s}ds)$ in $C([0,T];\mathbb{R}^{n}))$, that is to say
				$$\lim_{n\rightarrow\infty}\sup_{t\in[0,T]}\Big|\mathcal{G}^{0}(\int_{0}^{t}\dot{h}^{\delta}_{s}ds)-\mathcal{G}^{0}(\int_{0}^{t}\dot{h}_{s}ds)\Big|=0.$$
			\end{proposition}
			
			\begin{proof}
				Let $\bar{X}^{h^{\delta}}=\mathcal{G}^{0}(\int_{0}^{t}\dot{h}_{s}^{\delta}ds)$, then $\bar{X}^{h^{\delta}}$ is the solution of the following equation
				\begin{equation}\label{sec4-eq4.7}
					\begin{aligned}
						&d\bar{X}_{t}^{h^{\delta}}=\bar{b}(t,\bar{X}_{t}^{h^{\delta}},\mathscr{L}_{\bar{X}_{t}^{0}})dt+\sigma(t,\mathscr{L}_{\bar{X}_{t}^{0}})\dot{u}_{t}^{\delta}dt.\\
					\end{aligned}
				\end{equation}
				If $h^{\delta}$ converges to an element $h$ in $S_{M}$ as $\delta\rightarrow0$, it suffices to show that $\bar{X}^{h^{\delta}}$ strongly converges to $\bar{X}^{h}$ in $C([0,T];\mathbb{R}^{n})$ as $\delta\rightarrow0$, which implies  $\bar{X}^{h^{\delta}}$ is relatively compact in $C([0,T];\mathbb{R}^{n})$.
				By the Arzel$\grave{a}$-Ascoli theorem, we just need to prove that $\{\bar{X}^{h^{\delta}}\}$ is uniformly bounded and equi-continuous in $C([0,T];\mathbb{R}^{n})$. For the boundness of $\{\bar{X}^{h^{\delta}}\}$, replacing $h$ in Lemma \ref{sec4-lemma4.1}  with $h^{\delta}$ and it is easy to obtain
				\begin{equation*}\label{sec4-eq4.8}
					\sup_{h\in S_{M}}\Big\{\sup_{t\in[0,T]}|\bar{X}_{t}^{h^{\delta}}|\Big\}\leq C_{T,H,M,|x|}.
				\end{equation*}
				Then, we need to verify the equi-continuous of $|\bar{X}_{t}^{h^{\delta}}|$ in $C([0,T];\mathbb{R}^{n})$. From the Eq. (\ref{sec4-eq4.7}), we can deduce that for $0\leq s<t\leq T$,
				\begin{equation*}\label{sec4-eq4.9}
					\bar{X}_{t}^{h^{\delta}}-\bar{X}_{s}^{h^{\delta}}=\int_{s}^{t}\bar{b}(r,\bar{X}_{r}^{h^{\delta}},\mathscr{L}_{\bar{X}_{r}^{0}})dr+\int_{s}^{t}\sigma(r,\mathscr{L}_{\bar{X}_{r}^{0}})\dot{u}_{r}^{\delta}dr.
				\end{equation*}
				By  assumptions ($\mathbf{H1}$), we can have
				\begin{equation*}\label{sec4-eq4.10}
					\begin{aligned}
						\Big	|\int_{s}^{t}\bar{b}(r,\bar{X}_{r}^{h^{\delta}},\mathscr{L}_{\bar{X}_{r}^{0}})dr\Big|\leq&\int_{s}^{t}|\bar{b}(r,\bar{X}_{r}^{h^{\delta}},\mathscr{L}_{\bar{X}_{r}^{0}})-\bar{b}(r,0,\delta_{0})|dr+\int_{s}^{t}|\bar{b}(r,0,\delta_{0})|dr\\
						\leq&C_{T}\Big(\int_{s}^{t}K(r^{2})a(|\bar{X}_{r}^{h^{\delta}}|^{2}+|\bar{X}_{r}^{0}|^{2}+1)dr\Big)^{\frac{1}{2}}+C_{T}\Big(\int_{s}^{t}K(r^{2})dr\Big)^{\frac{1}{2}}\\
						\leq&C_{T}\Big(K(t^{2})(1+a(1+C_{T,H,M,|x|}+\sup_{r\in[0,T]}|\bar{X}_{r}^{0}|^{2})\Big)^{\frac{1}{2}},
					\end{aligned}
				\end{equation*}
				which is bounded for $t\in[0,T]$.
				
				As for the boundness of the integral $\int_{s}^{t}\sigma(r,\mathscr{L}_{\bar{X}_{r}^{0}})\dot{u}^{\delta}_{r}dr$, we can have
				\begin{equation*}
					\begin{aligned}
						\int_{s}^{t}\sigma(r,\mathscr{L}_{\bar{X}_{r}^{0}})\dot{u}_{r}^{\delta}dr
						&\leq C_{T}\Big(\int_{s}^{t}\|\sigma(r,\mathscr{L}_{\bar{X}_{r}^{0}})\|^{2}\cdot|\dot{u}^{\delta}_{r}|^{2}dr\Big)^{\frac{1}{2}}\\
						&\leq C_{T}\Big(\int_{s}^{t}K(r^{2})(1+a(1+|\bar{X}_{r}^{0}|^{2}))dr\Big)^{\frac{1}{2}}\cdot(\|u_{t}^{\delta}\|^{2}_{\mathcal{H}_{H}})^{\frac{1}{2}}\\
						&\leq C_{T,a,M}(1+\sup_{r\in[s,t]}|\bar{X}_{r}^{0}|^{2})^{\frac{1}{2}}.
					\end{aligned}
				\end{equation*}
				Thus we have proved that $\{\bar{X}^{h^{\delta}}\}$ is relatively compact in $C([0,T];\mathbb{R}^{n})$, which implies that for any sequence of $\{\bar{X}^{h^{\delta}}\}$, we can extract a further subsequence such that $\bar{X}^{h^{\delta}}$ converges to some $\hat{\bar{X}}$ in $C([0,T];\mathbb{R}^{n})$. In the following discussion, we aim to show that $\hat{\bar{X}}=\bar{X}^{h}$. By a small modification in Fan et al. (\cite{Fan}, Proposition 4.5), we can easily come to this conclusion. This completes the proof.
			\end{proof}

			\begin{proposition}\label{sec4-prop4.3}
				Under the assumptions in Theorem \ref{sec3-th3.5}, let $\{h^{\delta}\}_{\delta>0}\subset\mathcal{A}_{M}$ for any $M<\infty$. Then for any $\epsilon_{0}>0$, we have
				$$\lim_{\delta\rightarrow0}\mathbb{P}\Big(d(\mathcal{G}^{\delta}(\delta^{H}B^{H}_{\cdot}+\int_{0}^{\cdot}\dot{h}_{s}^{\delta}ds),\mathcal{G}^{0}(\int_{0}^{\cdot}\dot{h}_{s}^{\delta}ds))>\epsilon_{0}\Big)=0,$$
				where $d(\cdot,\cdot)$ denotes the metric in the space $C([0,T];\mathbb{R}^{n})$.
			\end{proposition}
			
			\begin{proof}
				The proof of Proposition \ref{sec4-prop4.3} will be separated into two steps.
				
				\textbf{Step 1.}
				From the Eq. (\ref{sec1-eq1.7}) and Eq. (\ref{sec4-eq4.7}), we can get that $X^{\delta,\epsilon,h^{\delta}}_{t}-\bar{X}_{t}^{h^{\delta}}$ satisfies the equations
				\begin{equation*}\label{sec4-eq4.11}
					\left\{
					\begin{aligned}
						d(X_{t}^{\delta,\epsilon,h^{\delta}}-\bar{X}_{t}^{h^{\delta}})=&\Big[b(t,X_{t}^{\delta,\epsilon,h^{\delta}},\mathscr{L}_{X_{t}^{\delta,\epsilon}},Y_{t}^{\delta,\epsilon,h^{\delta}})-\bar{b}(t,\bar{X}_{t}^{h^{\delta}},\mathscr{L}_{\bar{X}_{t}^{0}})\Big]dt\\
						&+\Big[\sigma(t,\mathscr{L}_{X_{t}^{\delta,\epsilon}})-\sigma(t,\mathscr{L}_{\bar{X}_{t}^{0}})\Big]\dot{u}^{\delta}_{t}dt	+\delta^{H}\sigma(t,\mathscr{L}_{X_{t}^{\delta,\epsilon}})dB_{t}^{H},\\
						X_{0}^{\delta,\epsilon,h^{\delta}}-\bar{X}_{0}^{h^{\delta}}=0.&
					\end{aligned}
					\right.
				\end{equation*}
				Then, it follows that
				\begin{equation}\label{sec4-eq4.12}
					\begin{aligned}
						|X_{t}^{\delta,\epsilon,h^{\delta}}-\bar{X}_{t}^{h^{\delta}}|^{2}\leq&3\Big|\int_{0}^{t}(b(s,X_{s}^{\delta,\epsilon,h^{\delta}},\mathscr{L}_{X_{s}^{\delta,\epsilon}},Y_{s}^{\delta,\epsilon,h^{\delta}})-\bar{b}(s,\bar{X}_{s}^{h^{\delta}},\mathscr{L}_{\bar{X}_{s}^{0}}))ds\Big|^{2}\\
						&+3\Big|\int_{0}^{t}(\sigma(s,\mathscr{L}_{X_{s}^{\delta,\epsilon}})-\sigma(s,\mathscr{L}_{\bar{X}_{s}^{0}}))\dot{u}^{\delta}_{s}ds\Big|^{2}\\
						&+3\delta^{2H}\Big|\int_{0}^{t}\sigma(s,\mathscr{L}_{X_{s}^{\delta,\epsilon}})dB_{s}^{H}\Big|^{2}\\
						=&:I_{1}(t)+I_{2}(t)+I_{3}(t).
					\end{aligned}
				\end{equation}
				
				Using assumptions ($\mathbf{H1}$) and H\"{o}lder inequality, we can have
				\begin{equation}\label{sec4-eq4.13}
					\begin{aligned}
						\mathbb{E}\Big(\sup_{t\in[0,T]}	I_{2}(t)\Big)
						&=3\mathbb{E}\Big(\sup_{t\in[0,T]}\Big|\int_{0}^{t}(\sigma(s,\mathscr{L}_{X_{s}^{\delta,\epsilon}})-\sigma(s,\mathscr{L}_{\bar{X}_{s}^{0}}))\dot{u}_{s}^{\delta}ds\Big|^{2}\Big)\\
						&\leq C_{T,H}\mathbb{E}\Big(\sup_{t\in[0,T]}\int_{0}^{t}K(s^{2})\kappa(\mathbb{W}_{2}(\mathscr{L}_{X_{s}^{\delta,\epsilon}},\mathscr{L}_{\bar{X}_{s}^{0}})^{2})|\dot{u}_{s}^{\delta}|^{2}ds\Big)\\
						&\leq C_{T,H,a,M}\kappa\Big(\mathbb{E}(\sup_{t\in[0,T]}|X_{t}^{\delta,\epsilon}-\bar{X}_{t}^{0}|^{2})\Big).
					\end{aligned}
				\end{equation}
				
				For the term $I_{3}(t)$, applying the maximal inequality in Lemma \ref{sec2-lemma2.1}, we can get
				\begin{equation}\label{sec4-eq4.14}
					\begin{aligned} \mathbb{E}(\sup_{t\in[0,T]}I_{3}(t))\leq& C_{T,H}\delta^{2H}\int_{0}^{T}\|\sigma(s,\mathscr{L}_{X_{s}^{\delta,\epsilon}})\|^{2}ds\\ \leq&C_{T,H}\delta^{2H}\Big(\int_{0}^{T}\|\sigma(s,\mathscr{L}_{X_{s}^{\delta,\epsilon}})-\sigma(s,\delta_{0})\|^{2}ds+\int_{0}^{T}\|\sigma(s,\delta_{0})\|^{2}ds\Big)\\
						\leq&C_{T,H,\delta,a}\delta^{2H}\Big(\mathbb{E}(\sup_{t\in[0,T]}|X_{t}^{\delta,\epsilon}|^{2})+1\Big).
					\end{aligned}
				\end{equation}
				
				Then for the term $I_{1}(t)$, we have
				\begin{equation*}
					\begin{aligned}
						I_{1}(t)=&3\Big|\int_{0}^{t}[b(s,X_{s}^{\delta,\epsilon,h^{\delta}},\mathscr{L}_{X_{s}^{\delta,\epsilon}},Y_{s}^{\delta,\epsilon,h^{\delta}})-b(s(\Delta),X_{s(\Delta)}^{\delta,\epsilon,h^{\delta}},\mathscr{L}_{X_{s(\Delta)}^{\delta,\epsilon}},\bar{Y}_{s}^{\delta,\epsilon})]ds\\
						&+\int_{0}^{t}[b(s(\Delta),X_{s(\Delta)}^{\delta,\epsilon,h^{\delta}},\mathscr{L}_{X_{s(\Delta)}^{\delta,\epsilon}},\bar{Y}_{s}^{\delta,\epsilon})-\bar{b}(s(\Delta),X_{s(\Delta)}^{\delta,\epsilon,h^{\delta}},\mathscr{L}_{X_{s(\Delta)}^{\delta,\epsilon}})]ds\\
						&+\int_{0}^{t}[\bar{b}(s(\Delta),X_{s(\Delta)}^{\delta,\epsilon,h^{\delta}},\mathscr{L}_{X_{s(\Delta)}^{\delta,\epsilon}})-\bar{b}(s,X_{s}^{\delta,\epsilon,h^{\delta}},\mathscr{L}_{X_{s}^{\delta,\epsilon}})]ds\\
						&+\int_{0}^{t}[\bar{b}(s,X_{s}^{\delta,\epsilon,h^{\delta}},\mathscr{L}_{X_{s}^{\delta,\epsilon}})
						-\bar{b}(s,\bar{X}_{s}^{h^{\delta}},\mathscr{L}_{\bar{X}_{s}^{0}})]ds\Big|^{2}.\\
					\end{aligned}
				\end{equation*}
				Using the basic inequality, we can get
				\begin{equation}\label{sec4-eq4.16}
					\begin{aligned}
						I_{1}(t)\leq&12\Big[\Big|\int_{0}^{t}b(s,X_{s}^{\delta,\epsilon,h^{\delta}},\mathscr{L}_{X_{s}^{\delta,\epsilon}},Y_{s}^{\delta,\epsilon,h^{\delta}})-b(s(\Delta),X_{s(\Delta)}^{\delta,\epsilon,h^{\delta}},\mathscr{L}_{X_{s(\Delta)}^{\delta,\epsilon}},\bar{Y}_{s}^{\delta,\epsilon})ds\Big|^{2}\\
						&+\Big|\int_{0}^{t}b(s(\Delta),X_{s(\Delta)}^{\delta,\epsilon,h^{\delta}},\mathscr{L}_{X_{s(\Delta)}^{\delta,\epsilon}},\bar{Y}_{s}^{\delta,\epsilon})-\bar{b}(s(\Delta),X_{s(\Delta)}^{\delta,\epsilon,h^{\delta}},\mathscr{L}_{X_{s(\Delta)}^{\delta,\epsilon}})ds\Big|^{2}\\
						&+\Big|\int_{0}^{t}\bar{b}(s(\Delta),X_{s(\Delta)}^{\delta,\epsilon,h^{\delta}},\mathscr{L}_{X_{s(\Delta)}^{\delta,\epsilon}})-\bar{b}(s,X_{s}^{\delta,\epsilon,h^{\delta}},\mathscr{L}_{X_{s}^{\delta,\epsilon}})ds\Big|^{2}\\
						&+\Big|\int_{0}^{t}\bar{b}(s,X_{s}^{\delta,\epsilon,h^{\delta}},\mathscr{L}_{X_{s}^{\delta,\epsilon}})
						-\bar{b}(s,\bar{X}_{s}^{h^{\delta}},\mathscr{L}_{\bar{X}_{s}^{0}})ds\Big|^{2}\Big]\\
						=&:12(I_{11}(t)+I_{12}(t)+I_{13}(t)+I_{14}(t)).
					\end{aligned}
				\end{equation}
				
				Taking expectation of $I_{11}(t)$, $I_{13}(t)$ and $I_{14}(t)$, we have
				\begin{equation}\label{sec4-eq4.17}
					\begin{aligned}
						\mathbb{E}&\Big(\sup_{t\in[0,T]}(I_{11}(t)+I_{13}(t))\Big)\\
						\leq& C_{T,a}\int_{0}^{T}\kappa\Big(\mathbb{E}|X_{t}^{\delta,\epsilon,h^{\delta}}-X_{t(\Delta)}^{\delta,\epsilon,h^{\delta}}|^{2}+\mathbb{W}_{2}(\mathscr{L}_{X_{t}^{\delta,\epsilon}},\mathscr{L}_{X_{t(\Delta)}^{\delta,\epsilon}})^{2}+\mathbb{E}|Y_{t}^{\delta,\epsilon,h^{\delta}}-\bar{Y}_{t}^{\delta,\epsilon}|^{2}\Big)dt\\
						\leq &C_{T,a}(1+|x|^{2}+|y|^{2})\kappa\Big(\frac{\epsilon}{\delta}+\Delta^{2}\vee\Delta^{2H}+\kappa(\Delta^{2}\vee\Delta^{2H})\Big)
					\end{aligned}
				\end{equation}
				and
				\begin{equation}\label{sec4-eq4.18}
					\begin{aligned}
						\mathbb{E}\Big(\sup_{t\in[0,T]}I_{14}(t)\Big)\leq&C_{T,a}\int_{0}^{T}\kappa\Big(\mathbb{E}|X_{t}^{\delta,\epsilon,h^{\delta}}-\bar{X}_{t}^{h^{\delta}}|^{2}+\mathbb{W}_{2}(\mathscr{L}_{X_{t}^{\delta,\epsilon}},\mathscr{L}_{\bar{X}_{t}^{0}})^{2}\Big)dt.
					\end{aligned}
				\end{equation}
				Thus, substituting the inequalities from (\ref{sec4-eq4.13}) to (\ref{sec4-eq4.18}) into(\ref{sec4-eq4.12}) and taking expectation of  $|X_{t}^{\delta,\epsilon,h^{\delta}}-\bar{X}_{t}^{h^{\delta}}|^{2}$, it can be obtained that
				\begin{equation}\label{sec4-eq4.19}
					\begin{aligned}
						\mathbb{E}\Big(\sup_{t\in[0,T]}|X_{t}^{\delta,\epsilon,h^{\delta}}-\bar{X}_{t}^{h^{\delta}}|^{2}\Big)\leq&
						\mathbb{E}(\sup_{t\in[0,T]}I_{12}(t))+
						C_{T,H,a,M}\kappa\Big(\mathbb{E}(\sup_{t\in[0,T]}|X_{t}^{\delta,\epsilon}-\bar{X}_{t}^{0}|^{2})\Big)\\
						&+C_{T,H,\delta,a}\delta^{2H}(\mathbb{E}(\sup_{t\in[0,T]}|X_{t}^{\delta,\epsilon}|^{2})+1)\\
						&+C_{T,M}(1+|x|^{2}+|y|^{2})\kappa\Big(\frac{\epsilon}{\delta}+\Delta^{2}\vee\Delta^{2H}+\kappa(\Delta^{2}\vee\Delta^{2H})\Big)\\
						&+C_{T,a}\int_{0}^{T}\kappa\Big(\mathbb{E}(\sup_{r\in[0,t]}|X_{r}^{\delta,\epsilon,h^{\delta}}-\bar{X}_{r}^{h^{\delta}}|^{2})x+\mathbb{E}(\sup_{r\in[0,t]}|X_{r}^{\delta,\epsilon}-\bar{X}_{r}^{0}|^{2})\Big)dt.
					\end{aligned}
				\end{equation}
				
				\textbf{Step 2.}
				In this step, we will consider the term $I_{12}(t)$ using the time discretization method. By dividing the time interval, we can obtain
				\begin{equation*}\label{sec4-eq4.20}
					\begin{aligned}
						&\mathbb{E}[{\sup_{t\in[0,T]}}I_{12}(t)]\\&=\mathbb{E}\Big|\int_{0}^{T}b(s(\Delta),X_{s(\Delta)}^{\delta,\epsilon,h^{\delta}},\mathscr{L}_{X_{s(\Delta)}^{\delta,\epsilon}},\bar{Y}_{s}^{\delta,\epsilon})-\bar{b}(s(\Delta),X_{s(\Delta)}^{\delta,\epsilon,h^{\delta}},\mathscr{L}_{X_{s(\Delta)}^{\delta,\epsilon}})ds\Big|^{2}\\
						&\leq2\sup_{t\in[0,T]}\mathbb{E}\Big|\sum_{k=0}^{[t/\Delta]-1}\int_{k\Delta}^{(k+1)\Delta} [b(k\Delta,X_{k\Delta}^{\delta,\epsilon,h^{\delta}},\mathscr{L}_{X_{k\Delta}^{\delta,\epsilon}},\bar{Y}_{s}^{\delta,\epsilon})-\bar{b}(k\Delta,X_{k\Delta}^{\delta,\epsilon,h^{\delta}},\mathscr{L}_{X_{k\Delta}^{\delta,\epsilon}})]ds\Big|^{2}\\
						&\quad+2\sup_{t\in[0,T]}\mathbb{E}\Big|\int_{t(\Delta)}^{t}[b(t(\Delta),X_{t(\Delta)}^{\delta,\epsilon,h^{\delta}},\mathscr{L}_{X_{t(\Delta)}^{\delta,\epsilon}},\bar{Y}_{s}^{\delta,\epsilon})-\bar{b}(t(\Delta),X_{t(\Delta)}^{\delta,\epsilon,h^{\delta}},\mathscr{L}_{X_{t(\Delta)}^{\delta,\epsilon}})]ds\Big|^{2}\\
						&\leq C \sup_{t\in[0,T]}\mathbb{E}\Big([\frac{t}{\Delta}] \sum_{k=0}^{[t/\Delta]-1}\Big|\int_{k\Delta}^{(k+1)\Delta} [b(k\Delta,X_{k\Delta}^{\delta,\epsilon,h^{\delta}},\mathscr{L}_{X_{k\Delta}^{\delta,\epsilon}},\bar{Y}_{s}^{\delta,\epsilon})-\bar{b}(k\Delta,X_{k\Delta}^{\delta,\epsilon,h^{\delta}},\mathscr{L}_{X_{k\Delta}^{\delta,\epsilon}})]ds\Big|^{2}\Big)\\
						&\quad+C_{T}\Delta\sup_{t\in[0,T]}\mathbb{E}\int_{t(\Delta)}^{t}\Big(1+|X_{t(\Delta)}^{\delta,\epsilon,h^{\delta}}|^{2}+|\bar{Y}_{s}^{\delta,\epsilon}|^{2}+\mathbb{E}|X_{t(\Delta)}^{\delta,\epsilon}|^{2}\Big)ds\\
						&\leq \frac{C_{T}}{\Delta^{2}}\max_{0\leq k\leq[\frac{T}{\Delta}]-1}\mathbb{E}\Big|\int_{k\Delta}^{(k+1)\Delta} [b(k\Delta,X_{k\Delta}^{\delta,\epsilon,h^{\delta}},\mathscr{L}_{X_{k\Delta}^{\delta,\epsilon}},\bar{Y}_{s}^{\delta,\epsilon})-\bar{b}(k\Delta,X_{k\Delta}^{\delta,\epsilon,h^{\delta}},\mathscr{L}_{X_{k\Delta}^{\delta,\epsilon}})]ds\Big|^{2}\\
						&\quad+C_{T,H,\beta_{2}}(1+|x|^{2}+|y|^{2})\Delta^{2},\\
						&\leq C_{T}\frac{\epsilon^{2}}{\Delta^{2}}\max_{0\leq k\leq [\frac{T}{\Delta}]-1}\int_{0}^{\frac{\Delta}{\epsilon}}\int_{r}^{\frac{\Delta}{\epsilon}}
						\Phi_{k}(s,r)dsdr+C_{T,H,|x|,|y|,\beta_{2}}\Delta^{2},\\
					\end{aligned}
				\end{equation*}
				where $0\leq r\leq s \leq\frac{\Delta}{\epsilon} $
				\begin{equation*}\label{sec4-eq4.21}
					\begin{aligned}
						\Phi_{k}(s,r)=&\mathbb{E}\Big[\Big|\langle b(k\Delta,X_{k\Delta}^{\delta,\epsilon,h^{\delta}},\mathscr{L}_{X_{k\Delta}^{\delta,\epsilon}},\bar{Y}_{s\epsilon+k\Delta}^{\delta,\epsilon})-\bar{b}(k\Delta,X_{k\Delta}^{\delta,\epsilon,h^{\delta}},\mathscr{L}_{X_{k\Delta}^{\delta,\epsilon}}),\\
						&\qquad b(k\Delta,X_{k\Delta}^{\delta,\epsilon,h^{\delta}},\mathscr{L}_{X_{k\Delta}^{\delta,\epsilon}},\bar{Y}_{r\epsilon+k\Delta}^{\delta,\epsilon})-\bar{b}(k\Delta,X_{k\Delta}^{\delta,\epsilon,h^{\delta}},\mathscr{L}_{X_{k\Delta}^{\delta,\epsilon}})\rangle\Big|\Big].
					\end{aligned}
				\end{equation*}
				By the Lemma 3.8 in Shen et al. \cite{Shen2}, for $\beta_{0}\in(0,\beta_{1})$, we have the boundness of $	\Phi_{k}(s,r)$
				\begin{equation*}\label{sec4-eq4.22}
					\Phi_{k}(s,r)\leq C_{T}(1+|x|^{2}+|y|^{2})e^{-\frac{(s-r)\beta_{0}}{2}}.
				\end{equation*}
				
				Thus, we can obtain
				\begin{equation}\label{sec4-eq4.23}
					\mathbb{E}\Big({\sup_{t\in[0,T]}}I_{12}(t)\Big)\leq C_{T,\beta_{1}}(1+|x|^{2}+|y|^{2})(\frac{\epsilon^{2}}{\Delta^{2}}\cdot\frac{\Delta}{\epsilon}+\Delta^{2}).
				\end{equation}
				
				Then by (\ref{sec4-eq4.19}) and (\ref{sec4-eq4.23}), it implies that
				\begin{equation}\label{sec4-eq4.24}
					\begin{aligned}
						&	\mathbb{E}\Big(\sup_{t\in[0,T]}|X_{t}^{\delta,\epsilon,h^{\delta}}-\bar{X}_{t}^{h^{\delta}}|^{2}\Big)\\
						&	\leq \mathbb{E}\Big(\sup_{t\in[0,T]}|X_{t}^{\delta,\epsilon,h^{\delta}}-\bar{X}_{t}^{h^{\delta}}|^{2}\Big)+\mathbb{E}\Big(\sup_{t\in[0,T}|X_{t}^{\delta,\epsilon}-\bar{X}_{t}^{0}|^{2}\Big)\\
						&\leq C_{T,H,M,|x|,|y|,\beta_{1}}\Big(\frac{\epsilon}{\Delta}
						+\Delta^{2}+\kappa(\Delta^{2}\vee\Delta^{2H})+\delta^{2}+\kappa(\frac{\epsilon}{\delta}+\Delta^{2}\vee\Delta^{2H}+\kappa(\Delta^{2}\vee\Delta^{2H})\Big)\\
						&\quad+C_{T,H,M}\int_{0}^{T}\kappa\Big(\mathbb{E}(\sup_{r\in[0,t]}|X_{r}^{\delta,\epsilon,h^{\delta}}-\bar{X}_{r}^{h^{\delta}}|^{2})
						+\mathbb{E}(\sup_{r\in[0,t]}|X_{r}^{\delta,\epsilon}-\bar{X}_{r}^{0}|^{2})\Big)dt.
					\end{aligned}
				\end{equation}
				Therefore, letting $\Delta=\delta^{\frac{1}{2}}$ and using the fact that $\kappa(0)=0$, it is easy to see as $\delta\rightarrow0$
				\begin{equation*}
					\begin{aligned}
						Z(T)&\leq C_{T,H,M}\int_{0}^{T}\kappa\Big(Z(t)\Big)dt\\
						&\leq \epsilon_{1}+C_{T,H,M}\int_{0}^{T}\kappa\Big(Z(t)\Big)dt\\
					\end{aligned}
				\end{equation*}
				for every $\epsilon_{1}>0$, where $$Z(t)=\lim_{\delta\rightarrow0}\Big(\mathbb{E}(\sup_{r\in[0,t]}|X_{r}^{\delta,\epsilon,h^{\delta}}-\bar{X}_{r}^{h^{\delta}}|^{2})
				+\mathbb{E}(\sup_{r\in[0,t]}|X_{r}^{\delta,\epsilon}-\bar{X}_{r}^{0}|^{2})\Big).$$
				Therefore,  the Bihari's inequality yields
				$$Z(T)\leq G^{-1}\Big(G(\epsilon_{1})+C_{T,H,M}\Big)$$
				where $G(\epsilon_{1})+ C_{T,H,M}\in\textit{Dom}(G^{-1})$, $G^{-1}$ is the inverse function of $G(\cdot)$ and 
				$$G(v)=\int_{1}^{v}\frac{1}{\kappa(s)}ds,\qquad v>0.$$
				By assumptions \textbf{(H1)}, one sees that $\lim_{\epsilon_{1}\rightarrow0}G(\epsilon_{1})= -\infty$ and $\textit{Dom}(G^{-1})=(-\infty, G(\infty))$. Letting $\epsilon_{1}\rightarrow0$, it gives $Z(T)=0$, i.e.,
				$$\mathbb{E}(\sup_{r\in[0,t]}|X_{r}^{\delta,\epsilon,h^{\delta}}-\bar{X}_{r}^{h^{\delta}}|^{2})
				+\mathbb{E}(\sup_{r\in[0,t]}|X_{r}^{\delta,\epsilon}-\bar{X}_{r}^{0}|^{2})\rightarrow0,\quad \delta\rightarrow0.$$
			Applying Chebyshev's inequality and using Lemma \ref{sec4-lemma4.5}, for any $\epsilon_{0}>0$ we have
			\begin{equation}\label{sec4-eq4.25}
				\begin{aligned}
					&\mathbb{P}\Big(d(\mathcal{G}^{\delta}(\delta^{H}B^{H}_{t}+\int_{0}^{t}\dot{h}_{s}^{\delta}ds),\mathcal{G}^{0}(\int_{0}^{t}\dot{h}_{s}^{\delta}ds))>\epsilon_{0}\Big)\\
					&=\mathbb{P}\Big(\sup_{t\in[0,T]}|X_{t}^{\delta,\epsilon,h^{\delta}}-\bar{X}_{t}^{h^{\delta}}|>\epsilon_{0}\Big)\\
					&\leq \frac{\mathbb{E}\Big(\sup_{t\in[0,T]}|X_{t}^{\delta,\epsilon,h^{\delta}}-\bar{X}_{t}^{h^{\delta}}|^{2}\Big)}{\epsilon_{0}^{2}}\rightarrow0,\quad as\quad\delta\rightarrow0.
				\end{aligned}
			\end{equation}
			This completes the proof.
			
		\end{proof}
		
		\textbf{Proof of Theorem 3.5.}
		Combining Proposition \ref{sec4-prop4.2} and Proposition \ref{sec4-prop4.3} and by the Lemma \ref{sec3-lemma3.5}, it is easy to see $\{X^{\delta,\epsilon}\}_{\delta>0}$ satisfies the Laplace principle, which is equivalent to the LDP in $C([0,T];\mathbb{R}^{n})$ with a good rate function $I$ defined in (\ref{sec3-eq3.0}).

		\bigskip
		
		$\begin{array}{cc}
			\begin{minipage}[t]{1\textwidth}
				{\bf Guangjun Shen}\\
				Department of Mathematics, Anhui Normal University, Wuhu 241002, China\\
				\texttt{gjshen@163.com}
			\end{minipage}
			\hfill
		\end{array}$

		$\begin{array}{cc}
			\begin{minipage}[t]{1\textwidth}
				{\bf Huan Zhou}\\
				Department of Mathematics, Anhui Normal University, Wuhu 241002, China\\
				\texttt{zhouhuan\_1997@163.com}
			\end{minipage}
			\hfill
		\end{array}$
	
	$\begin{array}{cc}
\begin{minipage}[t]{1\textwidth}
{\bf Jiang-Lun Wu}\\
Department of Mathematics, Computational Foundry,
 Swansea University \\ Swansea, SA1 8EN, UK\\
\texttt{j.l.wu@swansea.ac.uk} \\ 
\hfill
Current address: Faculty of Science and Technology, 
BNU-HKBU United \\ 
International College, Zhuhai 519087, China \\ 
\texttt{jianglunwu@uic.edu.cn}
\end{minipage}
\end{array}$

\end{document}